\documentclass[12pt]{article}
\usepackage[T1]{fontenc}
\usepackage[utf8]{inputenc}
\usepackage{lmodern}
\usepackage{amsmath,amssymb,amsthm}
\usepackage[margin=1in]{geometry}
\usepackage[expansion=false]{microtype}
\usepackage{enumitem}
\usepackage{aliascnt}
\usepackage{hyperref}
\usepackage[nameinlink,capitalize,noabbrev]{cleveref}
\hypersetup{
	colorlinks=true,
	linkcolor=blue,
	citecolor=blue,
	urlcolor=blue,
	pdftitle={Recurrence, Transience, and the Rate of Escape for Elephant Random Walks with Two Memory Channels},
	pdfauthor={Ngo Phuoc Nguyen Ngoc}
}
\newtheorem{theorem}{Theorem}[section]
\newaliascnt{proposition}{theorem}
\newtheorem{proposition}[proposition]{Proposition}
\aliascntresetthe{proposition}
\newaliascnt{lemma}{theorem}
\newtheorem{lemma}[lemma]{Lemma}
\aliascntresetthe{lemma}
\newaliascnt{corollary}{theorem}
\newtheorem{corollary}[corollary]{Corollary}
\aliascntresetthe{corollary}
\theoremstyle{remark}
\newaliascnt{remark}{theorem}
\newtheorem{remark}[remark]{Remark}
\aliascntresetthe{remark}
\theoremstyle{plain}
\crefname{theorem}{Theorem}{Theorems}
\crefname{proposition}{Proposition}{Propositions}
\crefname{lemma}{Lemma}{Lemmas}
\crefname{corollary}{Corollary}{Corollaries}
\crefname{remark}{Remark}{Remarks}
\newcommand{\Pp}{\mathbb{P}}
\newcommand{\E}{\mathbb{E}}
\newcommand{\R}{\mathbb{R}}
\newcommand{\F}{\mathcal{F}}
\newcommand{\G}{\mathcal{G}}
\newcommand{\Dact}{\mathfrak{D}}
\newcommand{\qeff}{q_{\mathrm{eff}}}
\newcommand{\1}{\mathbf{1}}
\newcommand{\sgn}{\operatorname{sgn}}
\newcommand{\as}{\mathrm{a.s.}}
\newcommand{\dd}{\,\mathrm{d}}
\title{Recurrence, Transience, and the Rate of Escape for Elephant Random Walks with Two Memory Channels}
\author{
Ngo Phuoc Nguyen Ngoc\\
\small Institute of Research and Development, Duy Tan University,
Da Nang, Vietnam\\
\small Faculty of Natural Sciences, Duy Tan University,
Da Nang, Vietnam\\
\small Email:
\href{mailto:ngopnguyenngoc@duytan.edu.vn}
{ngopnguyenngoc@duytan.edu.vn}
}
\date{}
\begin{document}
\maketitle
\begin{abstract}
We study the one-dimensional elephant random walk with two memory channels introduced by Saha [Phys.\ Rev.\ E \textbf{106}, L062105 (2022)] with memory parameter $p\in(0,1)$. Maulik, Roy and Sadhukhan [arXiv:2509.10225] proved recurrence for $p\le11/16$ and transience for $p>7/8$, leaving the range $11/16<p\le7/8$ open. We close this gap by proving that $p=11/16$ is the exact recurrence--transience threshold and by determining how fast the walk escapes from the origin throughout the transient regime. For $11/16<p<7/8$, we also show that the scaling limit obtained by Maulik, Roy and Sadhukhan is almost surely nonzero. At $p=7/8$, we prove that the walk is transient with zero asymptotic velocity and escapes at the scale $n/\sqrt{\log n}$.
\end{abstract}

\medskip


\medskip


\tableofcontents

\section{Introduction}
The elephant random walk (ERW) is a well-known example of a non-Markovian random walk with complete memory. In the classical one-dimensional model, let $W_1$ be a random sign, equal to $1$ or $-1$, each with probability $1/2$, and for $n\ge1$ choose, independently of the past and independently for different values of $n$,
\[
	U_{n+1}\sim\operatorname{Unif}\{1,\ldots,n\}, \qquad \sigma_{n+1}\in\{-1,1\}, \qquad \Pp(\sigma_{n+1}=1)=q,
\]
where $q\in(0,1)$ is the memory parameter. The walker then repeats the selected past increment or reverses it,
\[
	W_{n+1}:=\sigma_{n+1}W_{U_{n+1}},
\]
and its position is $E_n:=\sum_{j=1}^{n}W_j$. Since $E_n\equiv n\pmod2$ and a history of $n$ increments ending at $E_n=e$ contains $(n+e)/2$ increments equal to $+1$ and $(n-e)/2$ equal to $-1$,
\begin{equation}\label{eq:ERW-transition}
	\Pp(E_{n+1}=e+1\mid W_1,\ldots,W_n) =\frac12+\frac{(2q-1)e}{2n}.
\end{equation}
In particular $(n,E_n)$ is a Markov chain, the law after time $n$ depending on the past only through $n$ and $E_n$.

Sch\"utz and Trimper~\cite{SchutzTrimper2004} introduced the model, computed its first two moments, and identified the phase transition at $q=3/4$: the walk is diffusive for $q<3/4$, critical at $q=3/4$, and superdiffusive for
$q>3/4$. Further exact computations, and variants with modified memory rules, were carried out by, for example, Paraan and
Esguerra~\cite{ParaanEsguerra2006} and Cressoni, da~Silva and Viswanathan~\cite{CressoniSilvaViswanathan2007}.

A rigorous limit theory was then developed using urn and martingale methods. Baur and Bertoin~\cite{BaurBertoin2016} connected the ERW to a two-color P\'olya-type urn and applied the general urn theory of Janson~\cite{Janson2004}, while K\"ursten~\cite{Kursten2016} established a coupling between the ERW and bond percolation on random recursive trees. Martingale approaches were developed by Bercu~\cite{Bercu2018} and by Coletti, Gava and Sch\"utz~\cite{ColettiGavaSchutz2017a,ColettiGavaSchutz2017b}. These approaches yielded central limit theorems and laws of the iterated logarithm in the diffusive and critical regimes. In the superdiffusive regime, the suitably rescaled position converges almost surely to a non-Gaussian random variable. Gaussian fluctuation results for this regime were later obtained by Kubota and  Takei~\cite{KubotaTakei2019}. Gu\'erin, Laulin and Raschel~\cite{GuerinLaulinRaschel} proved that the limiting random variable in the superdiffusive regime has a density with full support on $\R$. Gu\'erin, Laulin, Raschel and Simon~\cite{GuerinLaulinRaschelSimon} went on to study its moments, tail behavior, and regularity properties in greater detail. Further results include Cram\'er moderate deviations~\cite{FanHuMa2021} and rates of convergence in the central limit theorem for ERWs with random step sizes~\cite{DedeckerFanHuMerlevede2023}. Multidimensional versions were studied by Bercu and Laulin~\cite{BercuLaulin2019} and Bertenghi~\cite{Bertenghi2022}. The ERW is also closely related to step-reinforced random walks~\cite{Bertoin2021,Businger2018}.

The recurrence and transience of the ERW have also been studied extensively. Bertoin~\cite{Bertoin2022} studied returns to the origin and first-return times in the diffusive regime. Coletti and Papageorgiou~\cite{ColettiPapageorgiou2021} investigated recurrence and transience in one dimension, and Qin~\cite{Qin2025} later completed the recurrence--transience classification in all dimensions. For $d=1,2$, Qin identified the threshold
\[
	q_d=\frac{2d+1}{4d},
\]
whereas for $d\ge3$ the walk is transient for every value of the memory parameter. Curien and  Laulin~\cite{CurienLaulin2024} subsequently gave a short proof of recurrence for the planar ERW in the diffusive regime. Our comparison argument adapts the monotone coupling used by Qin~\cite{Qin2025} to compare one-dimensional ERWs with different memory parameters.

Several variants of the ERW allow zero increments or modify the memory rule. Kumar, Harbola and Lindenberg~\cite{KumarHarbolaLindenberg2010} studied a model exhibiting subdiffusive, diffusive and superdiffusive regimes. The ERW with stops was analyzed by Bercu~\cite{Bercu2022stops} and later in higher
dimensions~\cite{Bercu2025stops}; see also Miyazaki and Takei~\cite{MiyazakiTakei2020}. Models with restricted or gradually fading memory were studied by Gut and Stadtm\"uller~\cite{GutStadtmuller2021,GutStadtmuller2022}, Laulin~\cite{Laulin2022}, and Chen and Laulin~\cite{ChenLaulin2023}. In ERWs with stops, zero is included directly among the possible increments. The two-channel model considered here produces zero increments through a different mechanism.

\subsection{The model and previous results}

Saha~\cite{Saha2022} introduced random walks with several memory channels, in which more than one selected past increment contributes to the next move. We consider the two-channel model. Let $S_0=0$ and let $X_1,X_2$ be independent random signs, each equal to $1$ or $-1$ with probability $1/2$. For $n\ge2$, choose, independently of the past and independently for different values of $n$,
\[
	U_{1,n+1},U_{2,n+1}\stackrel{\mathrm{iid}}{\sim} \operatorname{Unif}\{1,\ldots,n\}, \qquad \sigma_{1,n+1},\sigma_{2,n+1}\in\{-1,1\},
\]
where $\sigma_{1,n+1}$ and $\sigma_{2,n+1}$ are independent of each other and of the selected indices, and satisfy
\[
	\Pp(\sigma_{i,n+1}=1)=p, \qquad \Pp(\sigma_{i,n+1}=-1)=1-p.
\]
The two memory channels return the values
\[
	Y_{i,n+1}:=\sigma_{i,n+1}X_{U_{i,n+1}}, \qquad i=1,2,
\]
and the next increment is
\begin{equation}\label{eq:model}
	X_{n+1}:=\sgn(Y_{1,n+1}+Y_{2,n+1}), \qquad \sgn(0):=0.
\end{equation}
The position is $S_n:=\sum_{j=1}^nX_j$. Since the two selected values may cancel, the walk can have zero increments. We therefore also introduce
\[
	A_n:=\sum_{j=1}^n\1_{\{X_j\ne0\}},
\]
which counts the number of nonzero steps up to time $n$; note that $A_2=2$, since $X_1,X_2\in\{-1,1\}$. The triple $(n,S_n,A_n)$ contains all the information that is needed for the one-step transition probabilities; see \cref{rem:markov}.  Throughout the paper,
\[
	\F_n:=\sigma(X_1,\ldots,X_n), \qquad \mu:=2p-1=\E[\sigma_{i,n+1}].
\]
Thus $\mu$ is the mean of the random sign used in each memory channel.

Using stochastic approximation, Maulik, Roy and Sadhukhan~\cite{MaulikRoySadhukhan2025} studied the asymptotic behavior of this model. Here and throughout the paper, recurrence means that the walk visits the origin infinitely often almost surely, while transience means that it visits the origin only finitely often almost surely. They proved recurrence for $0<p\le11/16$ and transience for $7/8<p<1$, leaving the range
\[
	\frac{11}{16}<p\le\frac78
\]
unresolved. They also established that $S_n/n\to0$ almost surely for $p<7/8$ and that $S_n/n$ converges almost surely to a nonzero limit for $p>7/8$. Their law of large numbers did not cover $p=7/8$, which was left open in \cite[Remark~2.3]{MaulikRoySadhukhan2025}. In addition, they identified a further transition in the fluctuation behavior at $p=(113+\sqrt{97})/128$.

The present paper completes the recurrence--transience classification and sharpens the asymptotic picture in the transient regime. We prove transience throughout $11/16<p\le7/8$ and give a self-contained proof of recurrence for $p<11/16$, showing that $p=11/16$ is the exact recurrence--transience threshold. We also determine the almost-sure growth exponent of $|S_n|$ throughout the transient regime. In the range $11/16<p<7/8$, we prove that the previously known scaling limit is almost surely nonzero. At the critical value $p=7/8$, we establish the missing law of large numbers and obtain the sharper asymptotics $|S_n|\asymp n/\sqrt{\log n}$, together with the asymptotic behavior of the proportion of nonzero steps up to order $1/\log n$. The complete recurrence--transience classification is stated in the following theorem.

\begin{theorem}\label{thm:main}
	Consider the ERW with two memory channels and memory parameter $p\in(0,1)$. The walk is recurrent if and only if
	\[
		p\le\frac{11}{16}.
	\]
	Equivalently, it is transient if and only if $p>11/16$.
\end{theorem}
The next theorem determines the growth exponent of $|S_n|$ throughout the transient regime. In the range $11/16<p<7/8$, it also shows that the corresponding scaling limit is nonzero almost surely.
\begin{theorem}\label{thm:escape-scale}
		Consider the ERW with two memory channels and memory parameter $p\in(0,1)$.
	\begin{enumerate}[label=\textup{(\roman*)}]
		\item If $11/16<p<1$, then $|S_n|\to\infty$ almost surely and
		      \begin{equation}\label{eq:escape-exponent-main}
			      \lim_{n\to\infty}\frac{\log |S_n|}{\log n} =\min\left\{\frac{8p-4}{3},1\right\} \quad\as.
		      \end{equation}
		\item If $11/16<p<7/8$, then there is a finite random variable $L_p$ such that
		      \begin{equation}\label{eq:nondegenerate-main}
			      \frac{S_n}{n^{(8p-4)/3}}\longrightarrow L_p \quad\as, \qquad \Pp(L_p=0)=0.
		      \end{equation}
	\end{enumerate}
\end{theorem}
Thus, for $11/16<p<7/8$, part~\textup{(i)} states equivalently that $|S_n|=n^{(8p-4)/3+o(1)}$ almost surely. We refer to $(8p-4)/3$ as the \emph{growth exponent} of $|S_n|$ in this regime. The convergence in part~\textup{(ii)} was obtained in \cite[Theorem~2.5(c)]{MaulikRoySadhukhan2025}; the new contribution is that this limit does not vanish. Our argument also reproves the convergence itself.

We next prove a law of large numbers valid for every $0<p\le7/8$. For $p<7/8$, this provides an alternative proof of the corresponding result in \cite{MaulikRoySadhukhan2025}. The case $p=7/8$, left open in \cite[Remark~2.3]{MaulikRoySadhukhan2025}, is settled here. Our result also resolves the corresponding open case in the stochastic approximation considered in \cite[Theorem~4.2]{MaulikRoySadhukhan2025}.

\begin{theorem}\label{thm:lln} For every $0<p\le7/8$,
	\begin{equation}\label{eq:lln}
		\frac{S_n}{n}\longrightarrow0, \qquad \frac{A_n}{n}\longrightarrow\frac23 \quad\as.
	\end{equation}
	In particular, at $p=7/8$ the walk has zero asymptotic velocity, and asymptotically two-thirds of its increments are nonzero.
\end{theorem}
At $p=7/8$, the law of large numbers does not determine the scale of $|S_n|$ or the rate at which $A_n/n$ approaches $2/3$. The next theorem gives precise asymptotics for both quantities.

\begin{theorem}\label{thm:critical-scale}
	Let $p=7/8$. There exist an almost surely finite random variable $\Theta$ and a random sign $\Xi\in\{-1,1\}$ with $\Pp(\Xi=1)=\Pp(\Xi=-1)=1/2$ such that, almost surely, $\sgn(S_n)=\Xi$ for all sufficiently large $n$ and
	\begin{equation}\label{eq:Theta}
		\frac{n^2}{S_n^2}-\frac{27}{64}\log n \longrightarrow\Theta .
	\end{equation}
	Consequently
	\begin{equation}\label{eq:critical-position-scale}
		\frac{S_n}{n} =\frac{\Xi}{\sqrt{\dfrac{27}{64}\log n+\Theta+o(1)}}, \qquad\text{and in particular}\qquad \sqrt{\log n}\,\frac{S_n}{n} \longrightarrow \frac{8\sqrt3}{9}\,\Xi \quad\as.
	\end{equation}
	Moreover
	\begin{equation}\label{eq:critical-activity-scale}
		\frac{A_n}{n}-\frac23 =\frac{2}{3}\cdot\frac1{\log n+\frac{64}{27}\Theta+o(1)}, \quad\text{and in particular}\quad \log n\left(\frac{A_n}{n}-\frac23\right) \longrightarrow\frac23 \quad\as.
	\end{equation}
	Equivalently, almost surely,
	\[
		A_n=\frac{2n}{3}+\frac{2n}{3\log n} -\frac{128\,\Theta}{81}\cdot\frac{n}{(\log n)^2} +o\!\left(\frac{n}{(\log n)^2}\right).
	\]
\end{theorem}
The walk at $p=7/8$ is therefore sub-ballistic: $|S_n|$ has order $n/\sqrt{\log n}$, while $A_n/n$ approaches $2/3$ at rate $1/\log n$. In particular, the growth exponent in \cref{thm:escape-scale} equals one at $p=7/8$, although $S_n/n\to0$. The random variable $\Theta$ gives more detailed information about these two limits. We do not determine its distribution; see \cref{sec:conclusion}.

\paragraph{AI disclosure.}
Generative AI tools were used during the preparation of this manuscript to improve the English presentation and readability. They were also used as an auxiliary tool in exploring possible refinements of the results, which contributed to strengthening \cref{thm:critical-scale} compared with its initial version. All mathematical statements, arguments, and proofs were subsequently checked and verified by the author, who takes full responsibility for the content of the manuscript.

\subsection{Proof strategy}

The main idea is to remove the zero increments and compare the resulting embedded walk with a classical ERW. Let $\tau_k$ be the time of the $k$th nonzero step and set
\[
    \widehat S_k:=S_{\tau_k}.
\]
Thus $\widehat S$ records the position of the original walk only at its nonzero steps. Although $\widehat S$ is not itself a classical ERW, the probability that its next step moves away from the origin has asymptotically the same form as the corresponding outward probability for a classical ERW. This comparison recovers the effective memory parameter introduced in Section~3: matching the two coefficients gives

$$
\qeff(p)=\frac12+\frac{2(2p-1)}{3}=\frac{8p-1}{6}.
$$
Thus $\qeff(p)$ is precisely the parameter for which the embedded walk has, asymptotically, the same outward bias as a classical ERW. The two thresholds appearing in the model are then related to the classical ERW through
\begin{equation}\label{eq:two-coincidences}
    \qeff(p)>\frac34 \quad\Longleftrightarrow\quad p>\frac{11}{16}, \qquad \qeff(p)=1 \quad\Longleftrightarrow\quad p=\frac78.
\end{equation}

The comparison is based on upper and lower couplings for $|\widehat S_k|$. Below $p=11/16$, an upper comparison with a recurrent classical ERW yields recurrence. Above $p=11/16$, a lower comparison with a superdiffusive classical ERW gives transience and polynomial lower bounds for the distance from the origin. In the range $11/16<p<7/8$, these bounds are strong enough to show that the almost-sure scaling limit
\[
    \frac{S_n}{n^{(8p-4)/3}}
\]
is nonzero. At $p=7/8$, the same comparison proves transience, while the precise scale of the walk requires a finer analysis of the recursion for the original process.

A second ingredient is the law of large numbers
\[
    \frac{S_n}{n}\longrightarrow0, \qquad \frac{A_n}{n}\longrightarrow\frac23
    \quad\as,
\]
valid for every $0<p\le7/8$. We prove it directly from the stochastic approximation for $(S_n/n,A_n/n)$ using a Lyapunov method. This result determines the asymptotic proportion of nonzero steps and is used in the comparison with the embedded walk.

Finally, at $p=7/8$, the law of large numbers does not determine the scale of the walk. A more precise analysis of the recursion yields the logarithmic correction and leads to the scale $n/\sqrt{\log n}$ for $|S_n|$, together with the corresponding $1/\log n$ correction for $A_n/n$.

The only results taken from \cite{MaulikRoySadhukhan2025} are recurrence at $p=11/16$ and the almost-sure convergence of $S_n/n$ to a nonzero limit for $p>7/8$. We also use the classical one-dimensional ERW results recalled in
\cref{prop:erw-recurrent,prop:erw-escape}, taken from \cite{Qin2025}. All other arguments are proved here.

\textbf{Organization.}
In \cref{sec:model} we derive the stochastic-approximation representation and establish the law of large numbers, \cref{thm:lln}. \Cref{sec:embedded-walk} introduces the embedded walk and develops the comparison with classical ERWs.
These comparison results are used in \cref{sec:main-proof} to prove the recurrence--transience classification in \cref{thm:main} and the growth-exponent results in \cref{thm:escape-scale}. \Cref{sec:critical-scale} treats the case $p=7/8$ and proves the asymptotics in \cref{thm:critical-scale}. Finally, \cref{sec:conclusion} concludes the paper.

\section{Stochastic approximation and the law of large numbers}\label{sec:model}
We first derive the conditional transition probabilities and write the resulting recursion in stochastic-approximation form. We then construct a Lyapunov function and use it to control the long-time behavior of the normalized process, which yields the law of large numbers \eqref{eq:lln} for $0<p\le7/8$. This result will be used in the next section.

\subsection{Transition probabilities and stochastic approximation}

We first compute the one-step transition probabilities. Among the first $n$ increments, the numbers of $+1$ and $-1$ steps are
\[
	\frac{A_n+S_n}{2} \qquad\text{and}\qquad \frac{A_n-S_n}{2},
\]
respectively. Conditionally on $\F_n$, for each $i\in\{1,2\}$,
\begin{align*}
	\Pp(Y_{i,n+1}=1\mid\F_n)
	&=\frac{p(A_n+S_n)+(1-p)(A_n-S_n)}{2n}
	=\frac{A_n+\mu S_n}{2n}
	=:a_n,\\
	\Pp(Y_{i,n+1}=-1\mid\F_n)
	&=\frac{(1-p)(A_n+S_n)+p(A_n-S_n)}{2n}
	=\frac{A_n-\mu S_n}{2n}
	=:b_n,\\
	\Pp(Y_{i,n+1}=0\mid\F_n)
	&=1-\frac{A_n}{n}
	=:c_n.
\end{align*}

The two channels are conditionally independent. By the definition \eqref{eq:model}, the next increment is $1$ precisely when the ordered pair of returned values is $(1,1)$, $(1,0)$, or $(0,1)$. Similarly, it is $-1$ for the ordered pairs $(-1,-1)$, $(-1,0)$, and $(0,-1)$. Hence
\begin{align}
    \Pp(X_{n+1}=1\mid\F_n)
    &=a_n^2+2a_nc_n,\\
    \Pp(X_{n+1}=-1\mid\F_n)
    &=b_n^2+2b_nc_n.
    \end{align}
    Consequently,
    \begin{align}
    \E(X_{n+1}\mid\F_n)
    &=\frac{\mu S_n(2n-A_n)}{n^2},
    \label{eq:drift-S}\\
    \Pp(X_{n+1}\ne0\mid\F_n)
    &=\frac{4nA_n-3A_n^2+\mu^2S_n^2}{2n^2}.
    \label{eq:activity}
\end{align}

\begin{remark}\label{rem:markov}
	The transition probabilities above depend on the past only through $(n,S_n,A_n)$. Hence $(S_n,A_n)_{n\ge2}$ is a time-inhomogeneous Markov chain, or equivalently, $(n,S_n,A_n)_{n\ge2}$ is a Markov chain. We shall use the standard strong Markov property at the stopping times introduced in
\cref{sec:embedded-walk}.
\end{remark}
Set
\begin{equation}\label{eq:scaled}
	x_n:=\frac{S_n}{n}, \qquad r_n:=\frac{A_n}{n}.
\end{equation}
Since $|S_n|\le A_n\le n$, the normalized process takes values in the compact triangle
\begin{equation}\label{eq:D}
	D:=\{(x,r)\in\R^2: |x|\le r\le1\}.
\end{equation}
Let $I_{n+1}:=\1_{\{X_{n+1}\ne0\}}$, so that $S_{n+1}=S_n+X_{n+1}$ and $A_{n+1}=A_n+I_{n+1}$. One has the exact recursions
\[
	x_{n+1}-x_n=\frac{1}{n+1}\bigl(X_{n+1}-x_n\bigr), \qquad r_{n+1}-r_n=\frac{1}{n+1}\bigl(I_{n+1}-r_n\bigr).
\]
From \eqref{eq:drift-S} and \eqref{eq:activity} in the coordinates
\eqref{eq:scaled}, one has
\[
	\E(X_{n+1}\mid\F_n)-x_n = x_n\left\{\mu(2-r_n)-1\right\}, \qquad \E(I_{n+1}\mid\F_n)-r_n =r_n-\frac32r_n^2+\frac{\mu^2}{2}x_n^2.
\]
Thus, we obtain the stochastic-approximation recursion
\begin{equation}\label{eq:SA}
	\binom{x_{n+1}}{r_{n+1}} = \binom{x_n}{r_n} +\frac{1}{n+1} \left[ H_p(x_n,r_n)+\Delta M_{n+1} \right],
\end{equation}
where
\[
	\Delta M_{n+1} 	:= 	\binom{X_{n+1}-\E(X_{n+1}\mid\F_n)} {I_{n+1}-\E(I_{n+1}\mid\F_n)}
\]
is a martingale-difference sequence, bounded since $|X_{n+1}|\le1$ and
$I_{n+1}\in\{0,1\}$, and
\begin{equation}\label{eq:H}
	H_p(x,r) := \binom{x\{\mu(2-r)-1\}} {r-\frac32r^2+\frac{\mu^2}{2}x^2}.
\end{equation}
This is the same stochastic-approximation recursion as in \cite{MaulikRoySadhukhan2025}, written in the coordinates $(x_n,r_n)$. We then establish the law of large numbers for $0<p\le7/8$ using a Lyapunov function.

\subsection{Proof of the law of large numbers}\label{subsec:critical-lln}

Set
\begin{equation}\label{eq:kappa}
	\kappa:=\frac{4\mu}{3}-1=\frac{8p-7}{3}.
\end{equation}
Thus $\kappa<0$ for $p<7/8$, $\kappa=0$ for $p=7/8$, and $\kappa>0$ for $p>7/8$. For $p\le7/8$, solving $H_p(x,r)=0$ on $D$ gives
\begin{equation}\label{eq:zeros}
	\{H_p=0\}\cap D =\left\{(0,0),\left(0,\frac23\right)\right\}.
\end{equation}
Indeed, the first coordinate of $H_p$ vanishes either when $x=0$, in which case the second coordinate forces $r\in\{0,2/3\}$, or when $r=r_0:=2-1/\mu$. In the latter case the second coordinate requires $\tfrac{\mu^2}{2}x^2=r_0\bigl(\tfrac32r_0-1\bigr)$, whose right-hand side is nonpositive because $\mu\le3/4$ gives $r_0\le2/3$; hence $x=0$ and $r_0\in\{0,2/3\}$, and no further zero arises. 

Consider
\begin{equation}\label{eq:V}
	V(x,r):=x^2+\frac83\left(r-\frac23\right)^2.
\end{equation}
Writing $y=r-2/3$, we have
\begin{equation}\label{eq:Vdot}
	\nabla V(x,r)\cdot H_p(x,r) = 2\kappa x^2(1+\mu y)-8ry^2.
\end{equation}
Since $1+\mu y\ge1/3$ and $\kappa\le0$ for $p\le7/8$,
\begin{equation}\label{eq:Vdot-bound}
	\nabla V(x,r)\cdot H_p(x,r) \le -\frac{2|\kappa|}{3}x^2 -8r\left(r-\frac23\right)^2 \le0.
\end{equation}

We now apply \eqref{eq:Vdot} to the recursion \eqref{eq:SA}, together with the auxiliary convergence lemma proved in
\cref{app:slow-variation}. This yields the following dichotomy.

\begin{proposition}\label{prop:two-limits}
	For every $0<p\le7/8$,
	\begin{equation}\label{eq:two-limits}
		(x_n,r_n)\longrightarrow (0,0) 	\quad\text{or}\quad (x_n,r_n)\longrightarrow  \left(0,\frac23\right) \qquad\as.
	\end{equation}
\end{proposition}

\begin{proof}
	The proof has three steps. First we show that $V(x_n,r_n)$ converges. Second we deduce that $r_n$ can approach only $0$ or $2/3$. Finally we show that, if $r_n\to2/3$, then necessarily $x_n\to0$.
    
	\smallskip
	\emph{Step 1.}
	Let $Z_n:=(x_n,r_n)$ and $V_n:=V(Z_n)$, and abbreviate the increment in
	\eqref{eq:SA} as
	\[
		Z_{n+1}-Z_n=\frac{1}{n+1}\bigl[H_p(Z_n)+\Delta M_{n+1}\bigr]=:\Delta_n, \qquad \|\Delta_n\|\le\frac{C_0}{n+1}, \qquad C_0:=\sup_{D}\|H_p\|+3 ,
	\]
	the bound holding with this deterministic $C_0$ because $D$ in \eqref{eq:D} is compact, $H_p$ is continuous, and $\|\Delta M_{n+1}\|\le\sqrt5$ for the Euclidean norm. Since $V$ is a quadratic polynomial, its second-order Taylor expansion is an exact identity with constant Hessian $\nabla^2V=\operatorname{diag}(2,16/3)$:
	\[
		V_{n+1}=V_n+\nabla V(Z_n)\cdot\Delta_n +\frac12\,\Delta_n^\top\nabla^2V\,\Delta_n .
	\]
	We bound these two terms separately. For the linear term, substitute $\Delta_n$ and split according to
	\eqref{eq:SA}:
	\[
		\begin{aligned}
			\nabla V(Z_n)\cdot\Delta_n  =\frac{1}{n+1}\,\nabla V(Z_n)\cdot H_p(Z_n) +\frac{1}{n+1}\,\nabla V(Z_n)\cdot\Delta M_{n+1}.
		\end{aligned}
	\]
	By \eqref{eq:Vdot-bound}, evaluated at $Z_n\in D$, the first summand is at
	most
	\[
		-\frac{1}{n+1} \left(\frac{2|\kappa|}{3}x_n^2+8\,r_n\left(r_n-\frac23\right)^2 \right).
	\]
	The second summand has vanishing conditional expectation, since $\nabla V(Z_n)$ is $\F_n$-measurable and $\E(\Delta M_{n+1}\mid\F_n)=0$. For the quadratic term, the Hessian being constant and $\Delta_n$ bounded as above,
	\[
		\left|\frac12\,\Delta_n^\top\nabla^2V\,\Delta_n\right| \le\frac{8}{3}\,\|\Delta_n\|^2 \le\frac{C}{(n+1)^2}, \qquad C:=\frac83C_0^2 .
	\]
	Taking conditional expectations in the Taylor identity and combining the
	three estimates yields
	\begin{equation}\label{eq:V-supermartingale}
		\E(V_{n+1}\mid\F_n) \le V_n -\frac{1}{n+1} \left(\frac{2|\kappa|}{3}x_n^2+8\,r_n\left(r_n-\frac23\right)^2 \right) +\frac{C}{(n+1)^2}.
	\end{equation}
	The error $C/(n+1)^2$ is summable, so the Robbins--Siegmund
	theorem~\cite{RobbinsSiegmund1971} gives
	\begin{equation}\label{eq:V-convergence}
		V_n\longrightarrow V_\infty \quad\text{and}\quad \sum_{n=2}^{\infty}\frac{1}{n+1} \left(\frac{2|\kappa|}{3}x_n^2+8\,r_n\left(r_n-\frac23\right)^2\right) <\infty
		\qquad\as.
	\end{equation}
    
	\smallskip
	\emph{Step 2.}
	Set $h(r)=r(r-2/3)^2$. Since
	\[
		|r_{n+1}-r_n|=\frac{|I_{n+1}-r_n|}{n+1}\le\frac1{n+1}
	\]
	and $h$ is Lipschitz on $[0,1]$, \cref{lem:slow-variation}, applied pathwise together with \eqref{eq:V-convergence}, implies 
	\begin{equation}\label{eq:h-to-zero}
		h(r_n)\longrightarrow 0 \quad\as.
	\end{equation}
	Thus
\[
	\operatorname{dist}\left(r_n,\left\{0,\frac23\right\}\right) \longrightarrow0.
\]
Moreover, $|r_{n+1}-r_n|\le1/(n+1)$, so $r_n$ cannot switch infinitely often between disjoint neighborhoods of $0$ and $2/3$. Therefore,
\begin{equation}\label{eq:r-dichotomy}
	r_n\longrightarrow0 \quad\text{or}\quad r_n\longrightarrow\frac23 \qquad\as.
\end{equation}
If $r_n\to0$, then $|x_n|\le r_n$ implies $x_n\to0$, and hence $Z_n\to(0,0)$.
    
	\smallskip
	\emph{Step 3.}
	Suppose first that $p<7/8$, so that $\kappa<0$. Then
	\eqref{eq:V-convergence} contains $\sum_n x_n^2/(n+1)<\infty$, while
	\[
		\left|x_{n+1}^2-x_n^2\right| =\left|x_{n+1}-x_n\right|\left|x_{n+1}+x_n\right| \le\frac{2}{n+1}\cdot2=\frac4{n+1},
	\]
	because $|x_{n+1}-x_n|=\frac{1}{n+1}|X_{n+1}-x_n| \le\frac{2}{n+1}$ and $|x_n|\le1$. Hence \cref{lem:slow-variation}, applied to $u_n=x_n^2$, gives $x_n\to0$; this holds irrespective of which alternative in \eqref{eq:r-dichotomy} occurs, and completes the proof when $p<7/8$. There remains the case $p=7/8$, where $\kappa=0$ and the sum in \eqref{eq:V-convergence} carries no information about $x_n$. So assume that $p=7/8$ and $r_n\to2/3$. From \eqref{eq:V-convergence}, $x_n^2=V_n-\frac83(r_n-2/3)^2$ converges to some $\ell\ge0$. We show that $\ell=0$. The second coordinate of \eqref{eq:SA} is 
	\begin{equation}\label{eq:r-recursion-critical}
		r_{n+1}-r_n =\frac1{n+1} \left(r_n-\frac32r_n^2+\frac{9}{32}x_n^2+\eta_{n+1}\right),
	\end{equation}
	where $(\eta_{n+1})$ is a bounded martingale-difference sequence. Since
	\[
		\sum_{n=2}^{\infty} \frac{\E(\eta_{n+1}^2\mid\F_n)}{(n+1)^2}<\infty \quad\as,
	\]
	the martingale series
	\[
		\sum_{n\ge2}\frac{\eta_{n+1}}{n+1}
	\]
	converges almost surely, see, e.g., \cite[Theorem 2.17]{HallHeyde1980}. If $\ell>0$, then
	\[
		d_n:=r_n-\frac32r_n^2+\frac{9}{32}x_n^2 \longrightarrow \frac{9}{32}\ell>0.
	\]
	Hence there are $c>0$ and $n_0$ such that $d_n\ge c$ for every $n\ge n_0$. Summing \eqref{eq:r-recursion-critical} from $n_0$ to $n'-1$ gives
	\[
		r_{n'}-r_{n_0} =\sum_{n=n_0}^{n'-1}\frac{d_n}{n+1} +\sum_{n=n_0}^{n'-1}\frac{\eta_{n+1}}{n+1} \ge c\sum_{n=n_0}^{n'-1}\frac1{n+1}+O(1),
	\]
	because the martingale series converges. The right-hand side tends to $+\infty$ as $n'\to\infty$, contradicting $0\le r_{n'}\le1$. Therefore $\ell=0$, and $Z_n\to(0,2/3)$.
\end{proof}

To complete the proof of the law of large numbers, we must rule out convergence to $(0,0)$. We first prove one fact that will also be used later for the embedded walk.
\begin{lemma}\label{lem:A-infinity}
	For every $p\in(0,1)$,
	\[
		A_n\longrightarrow\infty \quad\as.
	\]
\end{lemma}

\begin{proof}
Since $(A_n)_{n\ge2}$ is nondecreasing and integer-valued, it either diverges to $\infty$ or is eventually constant. It is therefore enough to show that, for every $k\ge2$, once the process reaches level $k$, it almost surely reaches level $k+1$.

Define
\[
    \tau_k:=\inf\{n\ge2:A_n=k\}.
\]
Fix $k\ge2$ and suppose that $\tau_k=m<\infty$. As long as no further
nonzero step occurs, $A_n=k$, and by \eqref{eq:activity},
\[
    \Pp(I_{n+1}=1\mid\F_n) = \frac{4kn-3k^2+\mu^2S_n^2}{2n^2} \ge \frac{k}{2n}, \qquad n\ge m.
\]
Hence, for every $\ell>m$,
\[
    \Pp(\tau_{k+1}>\ell\mid\F_m) \le \prod_{n=m}^{\ell-1} \left(1-\frac{k}{2n}\right).
\]
Since $\sum_n n^{-1}=\infty$, the product on the right tends to zero as $\ell\to\infty$. Therefore,
\[
    \Pp(\tau_{k+1}<\infty\mid\F_{\tau_k})=1
    \qquad\text{on }\{\tau_k<\infty\}.
\]
Since $\tau_2=2$, induction gives $\tau_k<\infty$ almost surely for every $k\ge2$, and hence
\[
    A_n\longrightarrow\infty \qquad\as.
\]
\end{proof}

We now combine the dichotomy in \cref{prop:two-limits} with the divergence of $A_n$ to prove \cref{thm:lln}.

\begin{proof}[Proof of \cref{thm:lln}]
	By \cref{prop:two-limits}, the only possible limits are $(0,0)$ and $(0,2/3)$. It remains to exclude the first one. Since
	\[
		\{(x_n,r_n)\to(0,0)\}\subseteq\{r_n\to0\},
	\]
	it suffices to prove that $\Pp(E)=0$, where
	\(
		E:=\{r_n\to0\}.
	\)
	Write $Q_n:=\Pp(I_{n+1}=1\mid\F_n)$. By \eqref{eq:activity},
	\[
		Q_n=\frac{4nA_n-3A_n^2+\mu^2S_n^2}{2n^2}.
	\]
	Since $0\le\mu^2S_n^2\le A_n^2$, we obtain
\begin{equation}\label{eq:Q-bounds}
	\frac{A_n}{n}\left(2-\frac{3A_n}{2n}\right)
	\le Q_n
	\le
	\frac{A_n}{n}\left(2-\frac{A_n}{n}\right)
	\le\frac{2A_n}{n}.
\end{equation}
These bounds hold for every $p\in(0,1)$.

	\smallskip
	
    \emph{Step 1.}
	Since
	$A_{n+1}=A_n+I_{n+1}$ with $I_{n+1}\in\{0,1\}$, we have $\log A_{n+1}-\log A_n=I_{n+1}\log(1+1/A_n)$; summing from $n=2$ to $n'-1$ and writing $I_{n+1}=Q_n+(I_{n+1}-Q_n)$ gives
	\begin{equation}\label{eq:log-A-decomp}
		\log A_{n'}-\log A_2 =\sum_{n=2}^{n'-1}Q_n\log\left(1+\frac1{A_n}\right)+M_{n'},
	\end{equation}
	where
	\[
		M_{n'}:=\sum_{n=2}^{n'-1}(I_{n+1}-Q_n)\log\left(1+\frac1{A_n}\right)
	\]
	is a martingale, the weight $\log(1+1/A_n)$ being $\F_n$-measurable.

	\medskip
	\emph{Step 2.}
	Since $I_{n+1}$ is conditionally Bernoulli with parameter $Q_n$,
	\[
		\E\!\left((I_{n+1}-Q_n)^2\mid\F_n\right) =Q_n(1-Q_n)\le Q_n.
	\]
	Therefore, using $\log(1+x)\le x$ for $x\ge0$, the upper bound in \eqref{eq:Q-bounds}, and $A_n\ge2$, we obtain
	\[
		\begin{aligned}
			\E\!\left((M_{n+1}-M_n)^2\mid\F_n\right)
			  \le
			Q_n\log^2\left(1+\frac1{A_n}\right) 
			  \le
			\frac{2A_n}{n}\frac1{A_n^2}         
			  \le\frac1n.
		\end{aligned}
	\]
	Hence
	\[
		\sum_{n\ge3} \frac{\E\!\left((M_{n+1}-M_n)^2\mid\F_n\right)} {(\log n)^2} \le \sum_{n\ge3}\frac1{n(\log n)^2} <\infty.
	\]
	The martingale convergence theorem \cite[Theorem~2.17]{HallHeyde1980}, applied to
	\[
		\sum_{n\ge3}\frac{M_{n+1}-M_n}{\log n},
	\]
	shows that this series converges almost surely. Kronecker's lemma therefore yields
	\begin{equation}\label{eq:M-small}
		\frac{M_n}{\log n}\longrightarrow0 	\quad\as.
	\end{equation}
    
	\smallskip
    
	\emph{Step 3.}
	Suppose that $\Pp(E)>0$. By \cref{lem:A-infinity}, $A_n\to\infty$ on $E$, so we may fix $\varepsilon>0$ and constants
	$c_1>c_0>1$ with $2-\tfrac32\varepsilon>c_1>c_0>1$. On $E$ we have $r_n\le\varepsilon$ for all sufficiently large $n$, so $Q_n\ge c_1A_n/n$ by \eqref{eq:Q-bounds}; combining this with $A_n\log(1+1/A_n)\to1$ gives, for
	all large $n$ on $E$,
	\begin{equation}\label{eq:log-drift}
		Q_n\log\left(1+\frac1{A_n}\right)\ge\frac{c_0}{n}.
	\end{equation}
	Combining \eqref{eq:M-small} and \eqref{eq:log-drift} with \eqref{eq:log-A-decomp} and using
	$\sum_{n\le n'}1/n=\log n'+O(1)$, we obtain on $E$
	\[
		\log A_{n'}\ge c_0\log n'+o(\log n'), \qquad\text{hence}\qquad \liminf_{n'\to\infty}\frac{\log A_{n'}}{\log n'}\ge c_0>1.
	\]
	This contradicts the bound $A_{n'}\le n'$. Therefore $\Pp(E)=0$. The limit $(0,0)$ is impossible, so \cref{prop:two-limits} leaves only
	\[
		(x_n,r_n)\to\left(0,\frac23\right) \quad\as,
	\]
	which is exactly \eqref{eq:lln}.
\end{proof}
The law of large numbers has two consequences for the nonzero steps. The $k$th nonzero step occurs at a time asymptotically as $3k/2$, and the position at that time is $o(k)$. We now use these facts to determine the transition
probabilities after the zero increments are removed.

\section{The embedded walk and comparison with a classical ERW}\label{sec:embedded-walk}
Most of this section concerns the case $p>1/2$, so that $\mu=2p-1>0$; the hypothesis is explicitly stated in all results where it is needed. The simpler case $\mu\le0$ will be handled directly in the recurrence proof. For $k\ge2$, define the time of the $k$th nonzero step by
\begin{equation}\label{eq:tau}
	\tau_k:=\inf\{n\ge2:A_n=k\},
\end{equation}
already used in the proof of \cref{lem:A-infinity}. By \cref{lem:A-infinity}, $\tau_k<\infty$ almost surely. We define the \emph{embedded walk} $\widehat S$ by
\begin{equation}\label{eq:embedded-walk}
	\widehat S_k:=S_{\tau_k}, \qquad k\ge2.
\end{equation}
Equivalently, $\widehat S$ is obtained from $S$ by removing all zero increments. Notice that $(k,\widehat S_k)$ is not a Markov chain: even after the zero steps are deleted, the original dynamics still samples the entire past. Nevertheless, $\widehat S_k$ is a sum of exactly $k$ increments equal to $+1$ or $-1$, so that
\[
	\widehat S_k\equiv k\pmod 2, \qquad |\widehat S_k|\le k.
\]

\subsection{Outward transitions of the embedded walk}
    The comparison with a classical ERW rests on a single quantity: the \emph{outward probability}, namely, the probability that the next nonzero step moves the walk away from the origin rather than toward it. We compute it here and state the three properties of this quantity that are used in the comparison. Fix a time $n$ and suppose that $A_n=k$ and $S_n=s\ne0$. By \eqref{eq:activity}, the probability that the next step is nonzero is
    \begin{equation}\label{eq:Q-nks}
    	Q(n,k,s):= \Pp(X_{n+1}\ne0\mid\F_n) = \frac{\Dact(n,k,s)}{2n^2}, \text{ where }\Dact(n,k,s):=4kn-3k^2+\mu^2s^2.
    \end{equation}
    Since $k=A_n\le n$, we have $\Dact(n,k,s)\ge4kn-3k^2\ge kn>0$. Thus $Q(n,k,s)>0$, and it makes sense to condition on the next step being nonzero. We claim that the probability of then moving outward, that is, of
    increasing $|S|$, equals
    \begin{equation}\label{eq:outward-immediate}
    	\psi(n,k,s) 
    	:=\Pp\bigl(X_{n+1}=\sgn(s)\bigm|\F_n,\,X_{n+1}\ne0\bigr)
    	=\frac12+\frac{\mu|s|(2n-k)}{\Dact(n,k,s)}.
    \end{equation}
    Indeed, since $\{X_{n+1}=\sgn(s)\}\subseteq\{X_{n+1}\ne0\}$,
    \[
    	\psi(n,k,s)
    	=\frac{\Pp(X_{n+1}=\sgn(s)\mid\F_n)}{\Pp(X_{n+1}\ne0\mid\F_n)}.
    \]
    By \eqref{eq:drift-S} and \eqref{eq:activity}, the sum and difference of $\Pp(X_{n+1}=1\mid\F_n)$ and $\Pp(X_{n+1}=-1\mid\F_n)$ are
    \begin{align*}
    	 & \Pp(X_{n+1}=1\mid\F_n)+\Pp(X_{n+1}=-1\mid\F_n)=\frac{\Dact(n,k,s)}{2n^2},   \\
    	 & \Pp(X_{n+1}=1\mid\F_n)-\Pp(X_{n+1}=-1\mid\F_n)=\frac{\mu s(2n-k)}{n^2}.
    \end{align*}
    Adding and subtracting,
    \[
    	\Pp(X_{n+1}=\pm1\mid\F_n) =\frac12\left(\frac{\Dact(n,k,s)}{2n^2}\pm\frac{\mu s(2n-k)}{n^2}\right).
    \]
    Dividing by $Q(n,k,s)=\Dact(n,k,s)/(2n^2)$ therefore gives
    \[
    	\frac{\Pp(X_{n+1}=\pm1\mid\F_n)}{\Pp(X_{n+1}\ne0\mid\F_n)}  =\frac12\pm\frac{\mu s(2n-k)}{\Dact(n,k,s)}.
    \]
    If $s>0$ the outward direction is $+1$ and the upper sign applies, while if $s<0$ it is $-1$ and the lower sign applies; in both cases the correction equals $\mu|s|(2n-k)/\Dact(n,k,s)$, which proves \eqref{eq:outward-immediate}.
    Note that \eqref{eq:outward-immediate} was derived without any assumption on the sign of $\mu$, and that $|s|(2n-k)/\Dact(n,k,s)>0$. Hence the conditional outward probability is larger than $1/2$ when $\mu>0$, less than $1/2$ when $\mu<0$, and equal to $1/2$ when $\mu=0$. For the rest of this section we assume $\mu>0$, that is, $p>1/2$.
    Thus, conditional on moving, the walk has a bias away from the origin. This is the quantity that will be compared with the corresponding probability for a classical ERW. The bias in \eqref{eq:outward-immediate} depends on $(n,k,s)$ only through the two ratios
    \[
    	v:=\frac nk\ (\ge1), \qquad u:=\frac sk\ (|u|\le1).
    \]
    The first measures elapsed time per nonzero step, the second the position per nonzero step. This gives
    \begin{equation}\label{eq:C-def}
    	\psi(n,k,s) =\frac12+C_p(v,u)\,\frac{|s|}{k}, \qquad C_p(v,u):=\frac{\mu(2v-1)}{4v-3+\mu^2u^2}.
    \end{equation}
    This has the same form as the outward probability of a classical ERW with memory parameter $q$, which by \eqref{eq:ERW-transition} is 
    $\tfrac12+\tfrac{(2q-1)|y|}{2k}$, with $C_p(v,u)$ playing the role of $\tfrac{2q-1}{2}$. Thus the limiting value of $C_p$ determines the memory parameter of the classical ERW with which we should compare $\widehat S$. We need only three simple facts about $C_p$. First, on the domain $\{v\ge1,\ |u|\le1\}$ the denominator satisfies
    \[
    	4v-3+\mu^2u^2\ge1>0,
    \]
    so $C_p$ is continuous there. Moreover, $\mu(2v-1)>0$ and $2v-1\le4v-3+\mu^2u^2$ because $v\ge1$; hence
    \begin{equation}\label{eq:C-bounds}
    	0<C_p(v,u)\le \mu.
    \end{equation}
    Second, at the point $(v,u)=(3/2,0)$,
    \begin{equation}\label{eq:C-limit}
    	C_p\left(\frac32,0\right)=\frac{2\mu}{3},
    \end{equation}
    which gives the effective parameter $\qeff(p)=\tfrac12+\tfrac{2\mu}{3}$ introduced in \eqref{eq:qeff} below. This is the relevant point because, by \cref{thm:lln}, at nonzero-step times the ratios $v=n/k$ and $u=s/k$
    converge to $3/2$ and $0$, respectively. Third, for each fixed $u\in[-1,1]$, the map $v\mapsto C_p(v,u)$ is strictly decreasing, since
    \[
    	\frac{\partial}{\partial v}\,\frac{2v-1}{4v-3+\mu^2u^2} =\frac{2\left(\mu^2u^2-1\right)}{\left(4v-3+\mu^2u^2\right)^2}<0.
    \]
    Its value at $v=3/2$ and its limit at infinity are
    \[
    	C_p\left(\frac32,u\right)=\frac{2\mu}{3+\mu^2u^2}, \qquad \lim_{v\to\infty}C_p(v,u)=\frac{\mu}{2}.
    \]
    In particular, $C_p(3/2,u)\to2\mu/3$ as $u\to0$. Since $C_p(v,u)$ depends on $v$, its value at the next nonzero step is affected by the random number of preceding zero increments. The following lemma shows that, after averaging over these zero increments, the coefficient still converges to $2\mu/3$.
    
\begin{lemma}\label{lem:waiting-concentration}
	Let $t\ge k\ge2$ and let $(t,k,s)$ be an admissible state, so that 	$|s|\le k$ and $s\equiv k\pmod 2$. Denote by $\Pp_{t,k,s}$ and 	$\E_{t,k,s}$ probability and expectation for the process started at time $t$ from $(A_t,S_t)=(k,s)$, and let $N$ be the time immediately before the next nonzero step. For every $p\in(0,1)$,
	\begin{equation}\label{eq:waiting-weights}
		\Pp_{t,k,s}(N=n) =w_{t,k,s}(n) :=Q(n,k,s)\prod_{j=t}^{n-1} \bigl(1-Q(j,k,s)\bigr), \qquad n\ge t,
	\end{equation}
	and the weights $w_{t,k,s}(n)$ sum to one. If $p>1/2$, define
	\begin{equation}\label{eq:G-def}
		G(t,k,s):= \sum_{n=t}^{\infty} w_{t,k,s}(n) C_p\left(\frac nk,\frac sk\right).
	\end{equation}
	If $t_k\ge k$, $|s_k|\le k$,
	\[
		\frac{t_k}{k}\longrightarrow\frac32, \qquad \frac{s_k}{k}\longrightarrow0,
	\]
	then
		\begin{equation}\label{eq:C-average-limit}
			G(t_k,k,s_k)\longrightarrow\frac{2\mu}{3}.
		\end{equation}
\end{lemma}

\begin{proof}
	For $n\ge t$, the event $\{N=n\}$ means that
	\[
		X_{j+1}=0\quad\text{for }t\le j<n, \qquad X_{n+1}\ne0.
	\]
	During these zero steps, $(A_j,S_j)=(k,s)$. Multiplying the corresponding conditional probabilities gives
	\[
		\Pp_{t,k,s}(N=n) =Q(n,k,s)\prod_{j=t}^{n-1} \bigl(1-Q(j,k,s)\bigr),
	\]
	which proves \eqref{eq:waiting-weights}.

	The weights sum to one because
	\[
		\sum_{n\ge t}w_{t,k,s}(n) =1-\prod_{j\ge t}\bigl(1-Q(j,k,s)\bigr).
	\]
	Indeed, since $n\ge t\ge k$,
	\begin{equation}\label{eq:hazard-lower}
		Q(n,k,s) \ge\frac{4kn-3k^2}{2n^2} \ge\frac{k}{2n}.
	\end{equation}
	Hence $\sum_{n\ge t}Q(n,k,s)=\infty$, and the infinite product vanishes.

	To prove the final assertion, fix sequences $(t_k)$ and $(s_k)$ satisfying the assumptions of the lemma. By
	\eqref{eq:waiting-weights},
	\[
		G(t_k,k,s_k) =\E_{t_k,k,s_k}\left[ C_p\left(\frac Nk,\frac{s_k}{k}\right) \right].
	\]
	Since $s_k/k\to0$ and $C_p$ is continuous and bounded, it is enough to prove that
	\[
		\frac Nk\longrightarrow\frac32 \qquad\text{in probability}.
	\]

	Fix $\delta\in(0,1]$ and set $m_k:=\lfloor\delta k\rfloor$. Since $t_k/k\to3/2$, we have $t_k\le2k$ for all sufficiently large $k$. Thus, for $t_k\le n\le t_k+m_k$,
	\[
		\frac13\le\frac{k}{n}\le1.
	\]
	Writing $\varrho=k/n$, \eqref{eq:hazard-lower} gives
	\[
		Q(n,k,s_k) \ge\varrho\left(2-\frac32\varrho\right) \ge\frac12.
	\]
	Therefore,
    \[
        \Pp_{t_k,k,s_k}(N>t_k+m_k) = \prod_{j=t_k}^{t_k+m_k} \bigl(1-Q(j,k,s_k)\bigr) \le 2^{-(m_k+1)} \longrightarrow0.
    \]
	On the complementary event,
	\[
		\frac{t_k}{k} \le\frac Nk \le\frac{t_k}{k}+\delta.
	\]
	Since $t_k/k\to3/2$ and $\delta>0$ is arbitrary, this proves that $N/k\to3/2$ in probability. Consequently,
	\[
		C_p\left(\frac Nk,\frac{s_k}{k}\right) \longrightarrow 	C_p\left(\frac32,0\right) =\frac{2\mu}{3}
	\]
	in probability. The bound \eqref{eq:C-bounds} then gives convergence of the expectations, proving \eqref{eq:C-average-limit}.
\end{proof}

We now apply \cref{lem:waiting-concentration} to the embedded walk. The first result identifies the asymptotic coefficient in its outward transition probability.

\begin{proposition}\label{prop:embedded-outward}
	Let $p>1/2$ and assume that
	\begin{equation}\label{eq:activity-lln-assumption}
		\frac{S_n}{n}\to0, 	\qquad 	\frac{A_n}{n}\to\frac23 \quad\as.
	\end{equation}
	Let $\G_k:=\F_{\tau_k}$ and define
	\begin{equation}\label{eq:Bk-explicit}
		B_k := 	G\bigl(\tau_k,k,\widehat S_k\bigr) 	= \sum_{n=\tau_k}^{\infty} w_{\tau_k,k,\widehat S_k}(n)\, C_p\left(\frac nk,\frac{\widehat S_k}{k}\right).
	\end{equation}
	Then $B_k$ is $\G_k$-measurable and
	\begin{equation}\label{eq:Bk-limit}
		B_k\longrightarrow\frac{2\mu}{3} \quad\as.
	\end{equation}
	Moreover, on $\{\widehat S_k\ne0\}$,
	\begin{equation}\label{eq:Bk-def}
		\Pp\left(|\widehat S_{k+1}|=|\widehat S_k|+1 \,\middle|\, \G_k \right) = \frac12+B_k\frac{|\widehat S_k|}{k}.
	\end{equation}
	If $\widehat S_k=0$, then
	$|\widehat S_{k+1}|=1$ almost surely.
\end{proposition}

\begin{proof}
	Fix $k\ge2$ and condition on $\G_k$. Write
	\[
		t:=\tau_k, \qquad s:=\widehat S_k.
	\]
	By the strong Markov property at $\tau_k$, the conditional law of the process after time $t$ depends only on $(t,k,s)$. Since $t\ge k$, 	\cref{lem:waiting-concentration} gives
	\[
		\Pp(N=n\mid\G_k) =  w_{t,k,s}(n), \qquad n\ge t,
	\]
	where $N+1$ is the time of the next nonzero step.

	Since $\tau_k$ and $\widehat S_k$ are $\G_k$-measurable and $G$ is deterministic, $B_k=G(\tau_k,k,\widehat S_k)$ is $\G_k$-measurable.

	Suppose that $s\ne0$. On $\{N=n\}$, all increments between times $t+1$ and $n$ are zero, so that
	\[
		A_n=k, \qquad S_n=s.
	\]
	By \eqref{eq:C-def},
	\[
		\Pp\left(|\widehat S_{k+1}|=|s|+1 \,\middle|\, \G_k,N=n \right) = \frac12+ C_p\left(\frac nk,\frac sk\right)\frac{|s|}{k}.
	\]
	Averaging over $N$ and using $\sum_{n\ge t}w_{t,k,s}(n)=1$, we obtain
	\begin{align*}
		\Pp\left(|\widehat S_{k+1}|=|s|+1 \,\middle|\, \G_k \right)
		&= \sum_{n=t}^{\infty} w_{t,k,s}(n) \left[ \frac12+ C_p\left(\frac nk,\frac sk\right)\frac{|s|}{k} \right] \\
		&= 	\frac12+ G(t,k,s)\frac{|s|}{k} \\
		&= \frac12+B_k\frac{|s|}{k}.
	\end{align*}
	This proves \eqref{eq:Bk-def}.

	It remains to determine the limit of $B_k$. Since $A_{\tau_k}=k$ and $\tau_k\to\infty$, the assumed law of large numbers gives
	\[
		\frac{k}{\tau_k} = 	\frac{A_{\tau_k}}{\tau_k} \longrightarrow\frac23 \qquad\as,
	\]
	and hence
	\[
		\frac{\tau_k}{k}\longrightarrow\frac32 \qquad\as.
	\]
	Similarly,
	\[
		\frac{\widehat S_k}{k} 	=  \frac{S_{\tau_k}/\tau_k}{A_{\tau_k}/\tau_k} \longrightarrow0 \qquad\as.
	\]
	Applying the final assertion of \cref{lem:waiting-concentration} pathwise with $t_k=\tau_k$ and $s_k=\widehat S_k$, we obtain
	\[
		B_k = G(\tau_k,k,\widehat S_k) \longrightarrow\frac{2\mu}{3} \qquad\as.
	\]
	This proves \eqref{eq:Bk-limit}.

	Finally, if $\widehat S_k=0$, the next nonzero increment is either $1$ or $-1$, and therefore $|\widehat S_{k+1}|=1$ almost surely.
\end{proof}

    By  \eqref{eq:Bk-limit} and \eqref{eq:Bk-def}, the outward transition probability of the embedded walk therefore satisfies, on $\{\widehat S_k\ne0\}$,
    \begin{equation}\label{eq:embedded-out-asymptotic}
    	\Pp\left( |\widehat S_{k+1}|=|\widehat S_k|+1 \,\middle|\, \G_k \right) = \frac12+ \left(\frac{2\mu}{3}+o(1)\right) \frac{|\widehat S_k|}{k} \qquad\as.
    \end{equation}

    For comparison, if $(E_k)$ is a classical ERW with memory parameter $q$, then, conditionally on $E_k=y\ne0$,
    \begin{equation}\label{eq:ERW-out}
    	\Pp(|E_{k+1}|=|y|+1\mid E_k=y) = \frac12+\frac{(2q-1)|y|}{2k}.
    \end{equation}
    Comparing \eqref{eq:embedded-out-asymptotic} with \eqref{eq:ERW-out} shows that the two outward probabilities have the same asymptotic coefficient when
    \[
    	\frac{2q-1}{2}=\frac{2\mu}{3}.
    \]
    This motivates the definition
    \begin{equation}\label{eq:qeff}
    	\qeff(p) := \frac12+\frac{2\mu}{3}   = \frac{8p-1}{6},
    \end{equation}
    so that
    \begin{equation}\label{eq:coeff-match}
    	\frac{2\qeff(p)-1}{2} = \frac{2\mu}{3}.
    \end{equation}
    For $1/2<p\le7/8$, one has $\qeff(p)\in(1/2,1]$, and hence $\qeff(p)$ is a valid memory parameter for a classical ERW. For $p>7/8$, one has $\qeff(p)>1$; in that regime the quantity $\qeff(p)$ is only a convenient parametrization of the limiting outward coefficient.

The convergence of $B_k$ now gives upper and lower comparisons with classical ERWs whose memory parameters lie on either side of $\qeff(p)$.

\begin{proposition}\label{prop:embedded-comparison}
	Under the assumptions of \cref{prop:embedded-outward}, the following statements hold.
\begin{itemize}
    \item If $q_-\in(1/2,1)$ and $q_-<\qeff(p)$, then almost surely there exists a finite random $K_-$ such that, for every $k\ge K_-$ with $\widehat S_k\ne0$,
	\begin{equation}\label{eq:embedded-lower-prob}
		\Pp\left( |\widehat S_{k+1}|=|\widehat S_k|+1 \,\middle|\, 	\G_k \right) \ge \frac12+ \frac{(2q_--1)|\widehat S_k|}{2k}.
	\end{equation}
    \item If $\qeff(p)<1$ and $q_+\in(\qeff(p),1)$, then almost surely there exists a finite random $K_+$ such that, for every $k\ge K_+$ with $\widehat S_k\ne0$,
	\begin{equation}\label{eq:embedded-upper-prob}
		\Pp\left( |\widehat S_{k+1}|=|\widehat S_k|+1 \,\middle|\, \G_k \right) \le \frac12+ \frac{(2q_+-1)|\widehat S_k|}{2k}.
	\end{equation}
\end{itemize}
\end{proposition}

\begin{proof}
	If $q_-<\qeff(p)$, then by \eqref{eq:coeff-match},
	\[
		\frac{2q_--1}{2} < \frac{2\qeff(p)-1}{2} = \frac{2\mu}{3}.
	\]
	Since $B_k\to2\mu/3$ almost surely, there is almost surely a finite random $K_-$ such that
	\[
		B_k\ge\frac{2q_--1}{2}, \qquad k\ge K_-.
	\]
	Combining this with \eqref{eq:Bk-def} gives \eqref{eq:embedded-lower-prob}.

	Similarly, if $q_+>\qeff(p)$, then
	\[
		\frac{2\mu}{3} = \frac{2\qeff(p)-1}{2} 	< \frac{2q_+-1}{2}.
	\]
	Hence, almost surely,
	\[
		B_k\le\frac{2q_+-1}{2}
	\]
	for all sufficiently large $k$. Together with \eqref{eq:Bk-def}, this proves \eqref{eq:embedded-upper-prob}.
\end{proof}

\subsection{Coupling with classical ERWs}
We now give the coupling argument used to compare the absolute value of the embedded walk with that of a classical ERW. The argument is adapted from Qin's monotone coupling \cite[Proposition~1.9]{Qin2025}.

\begin{proposition}\label{prop:radial-coupling}
	Let $(R_k)_{k\ge K}$ be a nonnegative process adapted to $(\mathcal H_k)_{k\ge K}$ such that
	\(R_k\equiv k\pmod 2,\) $R_{k+1}=1$ whenever $R_k=0$, and otherwise $R_{k+1}=R_k\pm1$. On $\{R_k>0\}$, fix a version
	\[
		p_k := \Pp(R_{k+1}=R_k+1\mid\mathcal H_k).
	\]

	\begin{enumerate}[label=\textup{(\roman*)}]
		\item \emph{Lower comparison.} 
		Fix $q_{-}\in(1/2,1)$ and define
		\begin{equation}\label{eq:good-tail-event}
			\mathcal B_K^-(q_{-} ) := \bigcap_{k\ge K} \left( \{R_k=0\} \cup \left\{ p_k \ge \frac12+\frac{(2q_{-}-1)R_k}{2k} \right\} \right).
		\end{equation}
		On an extension of the probability space, one can construct a classical ERW with memory parameter $q_{-}$,
		denoted by $(E_k^{-,K})_{k\ge K}$, started from a deterministic history of $K$ increments, such that on $\mathcal B_K^-(q_-)$,
		\begin{equation}\label{eq:domination-on-good-tail}
			|E_k^{-,K}|\le R_k, \qquad k\ge K.
		\end{equation}

		\item \emph{Upper comparison.}
		Assume $R_K\le K$ and fix $q_+\in[1/2,1)$. Define
		\begin{equation}\label{eq:upper-tail-event}
			\mathcal B_K^{+}(q_+) := \bigcap_{k\ge K} \left( \{R_k=0\} \cup \left\{p_k \le \frac12+\frac{(2q_+-1)R_k}{2k} \right\} \right).
		\end{equation}
		Let $E^{+,K}$ be a classical ERW with memory parameter $q_+$ started from the all-positive history of length $K$. On an extension of the probability space, it can be coupled with $R$ so that, on $\mathcal B_K^{+}(q_+)$,
		\begin{equation}\label{eq:upper-domination}
			R_k\le |E_k^{+,K}|, \qquad k\ge K.
		\end{equation}
	\end{enumerate}

	In both cases the extension is obtained by adjoining auxiliary randomness independent of the original process. In particular, the law of $(R_k)_{k\ge K}$ is unchanged.
\end{proposition}

\begin{proof}
	On $\{R_k>0\}$, write
	\[
		O_{k+1} := \1_{\{R_{k+1}=R_k+1\}},
	\]
	so that
	\[
		\Pp(O_{k+1}=1\mid\mathcal H_k)=p_k.
	\]

   Let $(V_{k+1})_{k\ge K}$ be independent $\operatorname{Unif}(0,1)$ random variables, independent of the
original process. Define
\[
    U_{k+1}
    :=
    \begin{cases}
        p_kV_{k+1},
            & R_k>0,\ O_{k+1}=1,\\
        p_k+(1-p_k)V_{k+1},
            & R_k>0,\ O_{k+1}=0,\\
        V_{k+1},
            & R_k=0.
    \end{cases}
\]
We work on a probability-space extension carrying these auxiliary variables, and let $\widetilde{\mathcal H}_k$ denote the enlarged past up to time $k$. Since the auxiliary randomness is independent of the original process,
\[
    \Pp(O_{k+1}=1\mid\widetilde{\mathcal H}_k)=p_k.
\]
It follows that, conditionally on $\widetilde{\mathcal H}_k$, $U_{k+1}$ is uniform on $(0,1)$ and, on $\{R_k>0\}$,
\begin{equation}\label{eq:uniform-outward}
    O_{k+1} = \1_{\{U_{k+1}\le p_k\}} \qquad\text{a.s.}
\end{equation}

	\smallskip
	\emph{Lower comparison.}
	Choose a deterministic history of $K$ increments whose endpoint has the smallest possible absolute value:
	\[
		|E_K^{-,K}|
		=
		\begin{cases}
			0, & K\ \text{even},\\
			1, & K\ \text{odd}.
		\end{cases}
	\]
	Since $R_K\equiv K\pmod2$ and $R_K\ge0$,
	\[
		|E_K^{-,K}|\le R_K.
	\]

	Given $E_k^{-,K}\ne0$, use $U_{k+1}$ to move its absolute value outward whenever
	\[
		U_{k+1} \le \pi_k^-(E_k^{-,K}), \qquad \pi_k^-(e) := \frac12+\frac{(2q_{-}-1)|e|}{2k},
	\]
	and inward otherwise. If $E_k^{-,K}=0$, use an independent fair coin to choose the next sign.

	This gives a classical ERW with memory parameter $q_{-}$. Indeed, by \eqref{eq:ERW-transition}, if $E_k^{-,K}=e\ne0$ then the walk moves outward with probability
	\[
		\frac12+\frac{(2q_{-}-1)|e|}{2k} =\pi_k^-(e).
	\]

	We now prove \eqref{eq:domination-on-good-tail} by induction. Suppose
	\[
		|E_k^{-,K}|\le R_k.
	\]
	Since both quantities have the parity of $k$, either
	\[
		R_k-|E_k^{-,K}|\ge2
	\]
	or
	\[
		R_k=|E_k^{-,K}|.
	\]
	In the first case the order cannot be reversed in one step, since both absolute values change by at most one.

	Suppose therefore that
	\[
		R_k=|E_k^{-,K}|>0.
	\]
	On $\mathcal B_K^-(q_{-})$,
	\[
		p_k \ge \frac12+\frac{(2q_--1)R_k}{2k} 	= \pi_k^-(E_k^{-,K}).
	\]
	Thus, by \eqref{eq:uniform-outward}, whenever $|E^{-,K}|$ moves outward, $R$ also moves outward. Hence, the order cannot be reversed. If both values are zero, both absolute values equal one at the next time. Induction therefore yields
	\[
		|E_k^{-,K}|\le R_k, \qquad k\ge K,
	\]
	on $\mathcal B_K^-(q_{-})$.

	\smallskip
		\emph{Upper comparison.}
	Start $E^{+,K}$ from the all-positive history, so that $|E_K^{+,K}|=K\ge R_K$, and update it with the same variables $U_{k+1}$; its outward probability at time $k$ is
	\[
		\pi_k^+ := \frac12+\frac{(2q_+-1)|E_k^{+,K}|}{2k}.
	\]
	The induction is the one above with the two processes interchanged. Assume $R_k\le|E_k^{+,K}|$. A strict inequality is a gap of at least two by parity and cannot be reversed in one step; if both values are zero, both equal one at the next step. If $R_k=|E_k^{+,K}|>0$, then on $\mathcal B_K^{+}(q_+)$ we have $p_k\le\pi_k^+$, so an outward move of $R$ forces an outward move of $|E^{+,K}|$, and the order is again preserved. Hence
	\[
		R_k\le |E_k^{+,K}|, \qquad k\ge K,
	\]
	on $\mathcal B_K^{+}(q_+)$.

	Finally, the auxiliary variables are independent of the original process, so passing to the enlarged probability space does not change the law of $(R_k)_{k\ge K}$.
\end{proof}

The comparison also requires the classical superdiffusive limit for the ERW.

\begin{proposition}
	\label{prop:erw-escape}
	Let $q\in(3/4,1)$ and consider a classical ERW $(E_n)$ started from any deterministic initial history of $K\ge1$ increments. Then there exists a finite random variable $L$ such that
	\begin{equation}\label{eq:ERW-limit}
		\frac{E_n}{n^{2q-1}} \longrightarrow L \quad\as, \qquad \Pp(L=0)=0.
	\end{equation}
	In particular,
	\[
		|E_n|\longrightarrow\infty \quad\as.
	\]
\end{proposition}

Under the standard initial condition, the almost-sure convergence is the classical superdiffusive limit, while the nondegeneracy of the limit follows from \cite[Proposition~1.5]{Qin2025}. The extension to a deterministic finite
initial history follows by conditioning on that history, since the transition probabilities after time $K$ depend on the past only through $(K,E_K)$.

We now combine this result with the lower comparison to prove transience and a polynomial lower bound for our walk. The argument also covers the case $p=7/8$.

\begin{proposition}\label{prop:embedded-domination}
	Let $11/16<p\le7/8$ and assume \eqref{eq:activity-lln-assumption}. Then
	\begin{equation}\label{eq:embedded-escape}
		|\widehat S_k|\longrightarrow\infty \quad\as.
	\end{equation}
	Moreover, for every $0<\beta<\alpha(p)$, where 	$\alpha(p)=2\qeff(p)-1=(8p-4)/3$, there is an almost surely positive random constant $C_\beta$ such that
	\begin{equation}\label{eq:embedded-rate-q}
		|\widehat S_k|\ge C_\beta k^\beta
	\end{equation}
	for all sufficiently large $k$.
\end{proposition}

\begin{proof}
	By \eqref{eq:two-coincidences}, the hypothesis $p>11/16$ is equivalent to $\qeff(p)>3/4$; in particular $p>1/2$, so \cref{prop:embedded-comparison} applies. Throughout the proof, $q$ denotes a parameter with
	\[
		\frac34<q<\qeff(p),
	\]
	chosen separately for each of the two assertions.

	Fix any such $q$. For each deterministic $K\ge2$, write \(  \mathcal B_K^-:=\mathcal B_K^-(q), \) where $\mathcal B_K^-(q)$ is the event defined in \eqref{eq:good-tail-event} with
	\[
		R_k=|\widehat S_k|, \qquad \mathcal H_k=\G_k,
	\]
	and with the version of the conditional outward probabilities fixed in \eqref{eq:Bk-def}. By \cref{prop:embedded-comparison}, almost surely there exists a finite random index $K_-$ such that, for every $k\ge K_-$, on $\{\widehat S_k\ne0\}$,
	\[
		\Pp\left( |\widehat S_{k+1}|=|\widehat S_k|+1 \,\middle|\, \G_k \right) \ge \frac12+\frac{(2q-1)|\widehat S_k|}{2k}.
	\]
	At times when $\widehat S_k=0$, the alternative $\{R_k=0\}$ in the definition of $\mathcal B_K^-$ is automatically satisfied. Hence, almost every trajectory belongs to $\mathcal B_K^-$ for some finite $K$, and therefore
	\[
		\Pp\left( \bigcup_{K=2}^{\infty}\mathcal B_K^- \right)=1.
	\]
	Moreover,
	\[
		\mathcal B_K^-\subseteq\mathcal B_{K+1}^-,
	\]
	since increasing $K$ removes one condition from the defining intersection. Thus, by the continuity of probability from below,
	\begin{equation}\label{eq:one-sided-tail-prob}
		\Pp(\mathcal B_K^-) \longrightarrow \Pp\left( \bigcup_{K=2}^{\infty}\mathcal B_K^- \right) =1.
	\end{equation}
	Fix now a deterministic $K$. Apply \cref{prop:radial-coupling} after adjoining suitable auxiliary randomness independent of the original process and denote the probability measure on this extension by $\widetilde{\Pp}$. On $\mathcal B_K^-$, the coupling gives
	\begin{equation}\label{eq:embedded-lower-coupling}
		|E_k^{-,K}|\le|\widehat S_k|, \qquad k\ge K.
	\end{equation}
	By \cref{prop:erw-escape}, there exists a finite random variable $L_K$ such that
	\[
		\frac{E_k^{-,K}}{k^{2q-1}} \longrightarrow L_K \quad \widetilde{\Pp}\text{-a.s.}, \qquad \widetilde{\Pp}(L_K=0)=0.
	\]
	Consequently,
	\begin{equation}\label{eq:comparison-erw-asymptotic}
		|E_k^{-,K}| = |L_K|k^{2q-1}(1+o(1)) \quad \widetilde{\Pp}\text{-a.s.}, \qquad |L_K|>0 \quad \widetilde{\Pp}\text{-a.s.}
	\end{equation}
	Since $q>3/4$, we have $2q-1>0$, and hence
	\[
		|E_k^{-,K}|\longrightarrow\infty \quad \widetilde{\Pp}\text{-a.s.}
	\]
	Define the event
	\[
		\mathcal E := \left\{ |\widehat S_k|\longrightarrow\infty \right\}.
	\]
	Combining \eqref{eq:embedded-lower-coupling} with \eqref{eq:comparison-erw-asymptotic}, we obtain
	\[
		\widetilde{\Pp} \bigl(\mathcal E^c\cap\mathcal B_K^-\bigr) =0.
	\]
	Both $\mathcal E$ and $\mathcal B_K^-$ depend only on the original process. Since the added randomness is independent of it, the joint law of these two events is the same on the original and extended probability spaces. Therefore,
	\[
		\Pp\bigl(\mathcal E^c\cap\mathcal B_K^- \bigr) =0,
	\]
	and consequently
	\[
		\Pp(\mathcal E^c) = \Pp\bigl(\mathcal E^c\cap\mathcal B_K^-\bigr) + \Pp\bigl(\mathcal E^c\cap(\mathcal B_K^-)^c\bigr) \le \Pp((\mathcal B_K^-)^c).
	\]
	Letting $K\to\infty$ and using \eqref{eq:one-sided-tail-prob}, we obtain $\Pp(\mathcal E^c)=0$. Thus
	\[
		|\widehat S_k|\longrightarrow\infty \quad\as,
	\]
	which proves \eqref{eq:embedded-escape}.

	Fix $0<\beta<\alpha(p)$. Since $2\qeff(p)-1=\alpha(p)$, we may choose $q$ with
	\[
		\frac34<q<\qeff(p) 	\qquad\text{and}\qquad \beta<2q-1,
	\]
	and we form the events $\mathcal B_K^-$ with this $q$, so that \eqref{eq:one-sided-tail-prob}, \eqref{eq:embedded-lower-coupling} and \eqref{eq:comparison-erw-asymptotic} continue to hold. Define
	\[
		\mathcal D_\beta := \left\{ \liminf_{k\to\infty} \frac{|\widehat S_k|} {k^\beta}>0 \right\}.
	\]
	On $\mathcal B_K^-$, outside the $\widetilde{\Pp}$-null set associated with \eqref{eq:comparison-erw-asymptotic}, we have
	\[
		\frac{|\widehat S_k|}{k^\beta} \ge \frac{|E_k^{-,K}|}{k^\beta} = |L_K|k^{2q-1-\beta}(1+o(1)).
	\]
	Since $2q-1-\beta>0$ and $|L_K|>0$, the right-hand side tends to infinity. Hence
	\[
		\widetilde{\Pp} \bigl(\mathcal D_\beta^c\cap\mathcal B_K^- \bigr) =0.
	\]
	Again, both $\mathcal D_\beta$ and $\mathcal B_K^-$ depend only on the original process, so their joint law is unchanged on the extended space. Therefore,
	\[
		\Pp\bigl(\mathcal D_\beta^c\cap\mathcal B_K^- \bigr) =0, \qquad\text{and consequently}\qquad \Pp(\mathcal D_\beta^c) \le \Pp((\mathcal B_K^-)^c).
	\]
	Letting $K\to\infty$ and using \eqref{eq:one-sided-tail-prob}, we obtain $\Pp(\mathcal D_\beta)=1$. On $\mathcal D_\beta$, define
	\[
		C_\beta := \frac12 \left( 1\wedge \liminf_{k\to\infty} \frac{|\widehat S_k|}{k^\beta} \right),
	\]
	which is strictly positive there; on the null complement $\mathcal D_\beta^c$, set $C_\beta:=1/2$, so that $C_\beta$ is an almost surely positive random variable on the whole probability space. By the definition of the lower limit, on $\mathcal D_\beta$ there exists an almost surely finite random index $K_\beta$ such that
	\[
		|\widehat S_k| \ge 	C_\beta k^\beta, \qquad k\ge K_\beta,
	\]
	which proves \eqref{eq:embedded-rate-q}.
\end{proof}

\section{Recurrence, transience, and the rate of escape}\label{sec:main-proof}

The comparison developed in \cref{sec:embedded-walk} allows us to bound the absolute value of the embedded walk from above and below by suitable classical ERWs. We first use the upper comparison to prove recurrence by comparison with a recurrent classical ERW. We then use the lower comparison to prove transience for $11/16<p\le7/8$, the range left open in \cite{MaulikRoySadhukhan2025}.

\subsection{The recurrent regime}\label{subsec:recurrence}
We shall also use the corresponding recurrence result for the classical ERW.

\begin{proposition}
	\label{prop:erw-recurrent}
	Let $q\in[1/2,3/4]$ and consider a classical ERW $(E_n)$ started from any deterministic initial history of $K\ge1$ increments. Then
	\[
		E_n=0 \quad\text{for infinitely many }n \qquad\as.
	\]
\end{proposition}

Under the standard initial condition, this is \cite[Theorem~1.4]{Qin2025}. By symmetry, the same result holds for either choice of the first increment. Every prescribed finite initial history has positive probability under the corresponding ERW, and the transition probabilities after time $K$ depend on the past only through $(K,E_K)$. Conditioning on that history therefore gives the stated extension.

\begin{theorem}
\label{thm:recurrence}
	If $0<p<11/16$, then almost surely $S_n=0$ for infinitely many $n$; that is, the walk is recurrent.
\end{theorem}

\begin{proof}
We apply \cref{prop:radial-coupling} with
\[
    R_k=|\widehat S_k|, \qquad \mathcal H_k=\G_k.
\]
Its assumptions are satisfied since $|\widehat S_k|\equiv k\pmod2$, $|\widehat S_k|\le k$, and $|\widehat S_{k+1}|=1$ whenever $\widehat S_k=0$. We first show that, for every $p<11/16$, the absolute value of the embedded walk can be compared from above with that of a recurrent classical ERW.

\smallskip
\emph{Case $0<p\le1/2$.}
Here $\mu\le0$, and \eqref{eq:outward-immediate} gives
\[
    \psi(n,k,s)\le\frac12, \qquad n\ge k,\quad s\ne0.
\]
By the first assertion of \cref{lem:waiting-concentration}, conditionally on $\G_k$ the time immediately before the next nonzero step has weights $w_{\tau_k,k,\widehat S_k}(n)$. Hence, on $\{\widehat S_k\ne0\}$,

\begin{equation*}
    \Pp\left(|\widehat S_{k+1}|=|\widehat S_k|+1 \,\middle|\, \G_k \right)   = \sum_{n=\tau_k}^{\infty}    w_{\tau_k,k,\widehat S_k}(n)    \psi(n,k,\widehat S_k)    \le\frac12 .
\end{equation*}

Thus we may take $q_+=1/2$, and
\[
    \Pp\bigl(\mathcal B_K^+(1/2)\bigr)=1
\]
for every $K\ge2$.

\smallskip
\emph{Case $1/2<p<11/16$.}
By \eqref{eq:two-coincidences},
\[
    \qeff(p)<\frac34.
\]
Choose
\[
    \qeff(p)<q_+<\frac34.
\]
By \cref{thm:lln}, the assumptions of \cref{prop:embedded-comparison} are satisfied. Hence, almost surely, for all sufficiently large $k$,
\[
    \Pp\left( |\widehat S_{k+1}|=|\widehat S_k|+1 \,\middle|\, \G_k \right) \le \frac12+ \frac{(2q_+-1)|\widehat S_k|}{2k}
\]
on $\{\widehat S_k\ne0\}$. Therefore the events $\mathcal B_K^+(q_+)$ increase with $K$ and
\begin{equation}\label{eq:upper-tail-prob}
    \Pp\bigl(\mathcal B_K^+(q_+)\bigr)\longrightarrow1 .
\end{equation}

\smallskip
We now conclude in the same way in both cases. Fix a deterministic $K$ and apply the upper-comparison part of \cref{prop:radial-coupling}, denoting by $\widetilde\Pp$ the probability measure on the resulting extension. There exists a classical ERW $(E_k^{+,K})_{k\ge K}$ with memory parameter $q_+$ such that, on $\mathcal B_K^+(q_+)$,
\[
    |\widehat S_k| \le |E_k^{+,K}|, \qquad k\ge K.
\]
Write
\[
    \mathcal R:=\{S_n=0\ \text{for infinitely many } n\}.
\]
Since $q_+<3/4$, \cref{prop:erw-recurrent} gives $E_k^{+,K}=0$ for infinitely many $k$, $\widetilde\Pp$-almost surely. On $\mathcal B_K^+(q_+)$ the domination forces $\widehat S_k=0$ at each such $k$, and $\widehat S_k=S_{\tau_k}$ with $\tau_k<\infty$ almost surely by
\cref{lem:A-infinity}. Hence
\[
    \widetilde\Pp\bigl(\mathcal R^c\cap\mathcal B_K^+(q_+)\bigr)=0 .
\]
Both $\mathcal R$ and $\mathcal B_K^+(q_+)$ depend only on the original process, and the auxiliary randomness adjoined in \cref{prop:radial-coupling} is independent of it, so the joint law of these two events is the same under $\Pp$ and $\widetilde\Pp$. Therefore
$\Pp\bigl(\mathcal R^c\cap\mathcal B_K^+(q_+)\bigr)=0$, and consequently
\[
    \Pp(\mathcal R^c)  =\Pp\bigl(\mathcal R^c\cap\mathcal B_K^+(q_+)\bigr) +\Pp\bigl(\mathcal R^c\cap(\mathcal B_K^+(q_+))^c\bigr) \le\Pp\bigl((\mathcal B_K^+(q_+))^c\bigr).
\]
For $0<p\le1/2$ the right-hand side already vanishes for $K=2$, since $\Pp(\mathcal B_2^+(1/2))=1$. For $1/2<p<11/16$ it tends to $0$ as $K\to\infty$ by \eqref{eq:upper-tail-prob}. In both cases $\Pp(\mathcal R)=1$, which is the assertion.
\end{proof}

\begin{remark}\label{rem:endpoint}
	The argument above requires $q_+<3/4$, and therefore does not cover the endpoint $p=11/16$, where $\qeff(p)=3/4$: the convergence $B_k\to(2\qeff(p)-1)/2$ in \eqref{eq:Bk-limit} gives the required comparison only for a parameter strictly above the limit. There we use the recurrence proved in \cite[Proposition~2.8]{MaulikRoySadhukhan2025}.
\end{remark}

\subsection{The transient regime}\label{subsec:transience}
We begin with the range left open by Maulik, Roy and Sadhukhan \cite{MaulikRoySadhukhan2025}.

\begin{theorem}\label{thm:open-interval}
	If $\frac{11}{16}<p\le\frac78$, then
	\[
		|S_n|\longrightarrow\infty 	\quad\as.
	\]
	In particular, the walk is transient.
\end{theorem}

\begin{proof}
	Throughout the range
	\[
		\frac{11}{16}<p\le\frac78,
	\]
	we have
	\[
		\frac{S_n}{n}\longrightarrow0, \qquad \frac{A_n}{n}\longrightarrow\frac23 \quad\as.
	\]
	by \cref{thm:lln}. Thus, the assumptions of \cref{prop:embedded-domination} are satisfied. Moreover, by \eqref{eq:two-coincidences},
	\[
		\qeff(p)>\frac34 \quad\Longleftrightarrow\quad p>\frac{11}{16}.
	\]
	We may therefore choose
	\[
		\frac34<q<\qeff(p).
	\]
	The lower comparison in \cref{prop:embedded-domination} then gives
	\[
		|\widehat S_k|\longrightarrow\infty \quad\as.
	\]
	It remains only to return from the embedded walk to the original walk. Since $\widehat S_k$ records the position after the $k$th nonzero step,
	\[
		S_n=\widehat S_{A_n} \qquad n\ge 2.
	\]
	By \cref{lem:A-infinity}, $A_n\to\infty$ almost surely. Therefore,
	\[
		|S_n|=|\widehat S_{A_n}|\longrightarrow\infty \qquad\as.
	\]
	This proves the transience of the original walk.
\end{proof}

The same comparison gives more than transience: it yields a polynomial lower bound on the distance from the origin. This estimate will be again needed in the analysis at $p=7/8$.
\begin{proposition}\label{prop:lower-escape}
	Let $11/16<p\le7/8$ and set
	\[
		\alpha(p):=\frac{8p-4}{3}.
	\]
	For every $0<\beta<\alpha(p)$, there exist an almost surely positive random constant $C_\beta$ and an almost surely finite random time $N_\beta$ such that
	\begin{equation}\label{eq:poly-lower}
		|S_n|\ge C_\beta n^\beta, \qquad n\ge N_\beta.
	\end{equation}
\end{proposition}
\begin{proof}
	Fix $0<\beta<\alpha(p)$. By \cref{thm:lln}, the assumptions of \cref{prop:embedded-domination} hold throughout the stated range. Since
	\[
		2\qeff(p)-1=\alpha(p),
	\]
	we can choose
	\[
		\frac34<q<\qeff(p) \qquad\text{such that}\qquad \beta<2q-1.
	\]
	By \cref{prop:embedded-domination}, there exists an almost surely positive random constant $C$ such that, for all sufficiently large $k$,
	\[
		|\widehat S_k|\ge Ck^\beta \quad\as.
	\]
	On the other hand,
	\[
		\frac{A_n}{n}\longrightarrow\frac23 \quad\as.
	\]
	by \cref{thm:lln}. Since $S_n=\widehat S_{A_n}$ and $A_n\to\infty$, we obtain, for all sufficiently large $n$,
	\[
		|S_n| = |\widehat S_{A_n}| \ge 	C A_n^\beta.
	\]
	Moreover,
	\[
		\frac{A_n^\beta}{n^\beta} \longrightarrow \left(\frac23\right)^\beta \quad\as.
	\]
	Hence, after possibly decreasing the random constant, there exist an almost surely positive $C_\beta$ and an almost surely finite $N_\beta$ such that
	\[
		|S_n|\ge C_\beta n^\beta, \qquad n\ge N_\beta.
	\]
\end{proof}

The preceding lower bound is also strong enough to identify the exact normalization. In particular, it rules out a zero limit. 

\begin{proposition}\label{prop:nondegenerate-limit}
	Let $11/16<p<7/8$ and set
	\[
		\alpha:=\frac{8p-4}{3}.
	\]
	Then there is a finite random variable $L_p$ such that
	\begin{equation}\label{eq:nondegenerate-limit}
		\frac{S_n}{n^\alpha}\longrightarrow L_p \quad\as, \qquad \Pp(L_p=0)=0.
	\end{equation}
\end{proposition}

\begin{proof}
	Set
	\[
		x_n:=\frac{S_n}{n}, \qquad r_n:=\frac{A_n}{n}, \qquad y_n:=r_n-\frac23.
	\]

	We first prove that
	\begin{equation}\label{eq:y-harmonic-converges}
		\sum_{n=2}^{\infty}\frac{y_n}{n} \quad\text{converges almost surely.}
	\end{equation}
	Since $p<7/8$, we have $\kappa<0$ in \eqref{eq:kappa}. It follows from 	\eqref{eq:V-convergence} that
    
	\begin{equation}\label{eq:x-square-summable-subcritical} 
		\sum_{n=2}^{\infty}\frac{x_n^2}{n+1}<\infty \qquad\as.
	\end{equation}
	Moreover, the second term in \eqref{eq:V-convergence} implies
    \[
    	\sum_{n=2}^{\infty}\frac{r_ny_n^2}{n+1}<\infty 	\qquad\as.
    \]
    Since $r_n\to2/3$, we have $r_n\ge1/3$ for all sufficiently large $n$. Therefore,
    \begin{equation}\label{eq:y-square-summable-subcritical}
    	\sum_{n=2}^{\infty}\frac{y_n^2}{n+1}<\infty   	\qquad\as.
    \end{equation}
    
    Let
	\[
		I_{n+1}:=\1_{\{X_{n+1}\ne0\}}, 	\qquad \eta_{n+1}:= I_{n+1}-\E(I_{n+1}\mid\F_n).
	\]
	The second coordinate of \eqref{eq:SA} gives
	\begin{equation}\label{eq:y-recursion-subcritical}
		y_{n+1}-y_n = \frac1{n+1} \left( -y_n-\frac32y_n^2 +\frac{\mu^2}{2}x_n^2+\eta_{n+1}
		\right).
	\end{equation}
	Since $|\eta_{n+1}|\le1$,
	\[
		\sum_{n=2}^{\infty} \E\left( \left.
			\frac{\eta_{n+1}^2}{(n+1)^2} \right|\F_n \right) <\infty.
	\]
	Hence the martingale series
	\[
		\sum_{n=2}^{\infty}\frac{\eta_{n+1}}{n+1}
	\]
	converges almost surely.

	Rearranging \eqref{eq:y-recursion-subcritical}, we obtain
	\[
		\frac{y_n}{n+1} = -(y_{n+1}-y_n) -\frac{3y_n^2}{2(n+1)} +\frac{\mu^2x_n^2}{2(n+1)} +\frac{\eta_{n+1}}{n+1}.
	\]
	The first term on the right telescopes, the next two series converge absolutely by \eqref{eq:x-square-summable-subcritical} and \eqref{eq:y-square-summable-subcritical}, and the final series is the martingale series above. Therefore,
	\[
		\sum_{n=2}^{\infty}\frac{y_n}{n+1}
	\]
	converges almost surely. Since $|y_n|\le2/3$, the difference between $\sum_n y_n/n$ and $\sum_n y_n/(n+1)$ is absolutely summable. This proves \eqref{eq:y-harmonic-converges}.

	We now study the growth of $|S_n|$. Since $\alpha\in(1/2,1)$, choose $\beta$ such that
	\[
		\frac12<\beta<\alpha.
	\]
	By \cref{prop:lower-escape}, almost surely there exists $C_\beta>0$ such that
	\[
		|S_n|\ge C_\beta n^\beta
	\]
	for all sufficiently large $n$. Setting $R_n:=|S_n|$, we obtain 
    
	\begin{equation}\label{eq:inverse-square-summable}
		\sum_{\{n\ge2:R_n\ne0\}}\frac1{R_n^2}<\infty \qquad\as,
	\end{equation}
	because $2\beta>1$. In particular, $R_n\to\infty$. Since the increments belong to $\{-1,0,1\}$, a change of sign of $S_n$ forces a visit to the origin; hence the sign of $S_n$ is eventually constant.

	Let $\iota_n:=\sgn(S_n)$. For all sufficiently large $n$, $R_n\ge2$ and
	\[
		R_{n+1}=R_n+\iota_nX_{n+1}.
	\]
	Since
	$|\iota_n X_{n+1}/R_n|\le1/2$, Taylor's formula gives
	\begin{equation}\label{eq:log-distance-increment}
		\log R_{n+1}-\log R_n = \frac{\iota_nX_{n+1}}{R_n} +\rho_{n+1}, \qquad |\rho_{n+1}|\le\frac{C}{R_n^2},
	\end{equation}
	for some deterministic constant $C$. By \eqref{eq:inverse-square-summable}, $\sum_n|\rho_{n+1}|<\infty$ almost surely.

	Define
	\[
		D_{n+1}:= \1_{\{R_n\ge2\}}\frac{\iota_n}{R_n} \left( X_{n+1}-\E(X_{n+1}\mid\F_n) \right).
	\]
	Then $(D_{n+1})$ is a martingale-difference sequence and
	\[
		\sum_{n=2}^{\infty} \E(D_{n+1}^2\mid\F_n) \le \sum_{\{n:R_n\ge2\}}\frac1{R_n^2} <\infty \qquad\as.
	\]
	Therefore, $\sum_nD_{n+1}$ converges almost surely by the martingale convergence theorem, applied after localization; see \cite[Theorem~2.17]{HallHeyde1980}.

	Finally, \eqref{eq:drift-S} gives, for all sufficiently large $n$,
	\begin{align}
		\frac{\iota_n}{R_n}
		\E(X_{n+1}\mid\F_n)
		&=
		\frac{\mu(2-r_n)}{n} \notag\\
		&=
		\frac{4\mu}{3n}-\frac{\mu y_n}{n} = \frac{\alpha}{n}-\frac{\mu y_n}{n}.
		\label{eq:log-predictable-drift}
	\end{align}
	Combining this identity with \eqref{eq:log-distance-increment}, we obtain
	\[
		\log R_{n+1}-\log R_n = \frac{\alpha}{n} -\frac{\mu y_n}{n} +D_{n+1} +\rho_{n+1}
	\]
	for all sufficiently large $n$.

	The series involving $y_n$, $D_{n+1}$, and $\rho_{n+1}$ all converge almost surely. Since
	\[
		\sum_{k=1}^{n-1}\frac1k-\log n
	\]
	also converges, there exists an almost surely finite random variable $Z$ such that
	\[
		\log R_n-\alpha\log n \longrightarrow Z \qquad\as.
	\]
	Consequently,
	\[
		\frac{R_n}{n^\alpha}\longrightarrow e^Z>0 \qquad\as.
	\]

	Let $\Xi\in\{-1,1\}$ denote the eventual sign of $S_n$. Then 
	\[
		\frac{S_n}{n^\alpha} \longrightarrow \Xi e^Z=:L_p 	\qquad\as.
	\]
	Thus $L_p$ is finite and $|L_p|=e^Z>0$ almost surely, which proves \eqref{eq:nondegenerate-limit}.
\end{proof}

We now determine the growth exponent. For $11/16<p<7/8$, it follows from the nondegenerate almost-sure scaling limit proved above. At $p=7/8$, the polynomial lower bound, together with the trivial bound $|S_n|\le n$, yields the exponent $1$.

    \begin{corollary}\label{cor:escape-exponent-subballistic}
    	If $11/16<p\le7/8$, then
    	\begin{equation}\label{eq:escape-exponent-subballistic}
    		\lim_{n\to\infty}\frac{\log|S_n|}{\log n} =\frac{8p-4}{3} \quad\as.
    	\end{equation}
    \end{corollary}
\begin{proof}
	Suppose first $11/16<p<7/8$. By \cref{prop:nondegenerate-limit}, with $\alpha=(8p-4)/3$,
	\[
		\frac{|S_n|}{n^\alpha}\longrightarrow |L_p|>0 \quad\as.
	\]
	Taking logarithms gives $\log|S_n|/\log n\to\alpha$ almost surely.
    
	At $p=7/8$ we have $(8p-4)/3=1$. The bound $|S_n|\le n$ gives $\log|S_n|/\log n\le1$, so $\limsup_n\log|S_n|/\log n\le1$. For the lower bound, \cref{prop:lower-escape} applies at $p=7/8$ and yields, for each $\beta<1$, an almost surely positive $C_\beta$ and finite $N_\beta$ with $|S_n|\ge C_\beta n^\beta$ for $n\ge N_\beta$; hence
	\[
		\liminf_{n\to\infty}\frac{\log|S_n|}{\log n} \ge\liminf_{n\to\infty}\frac{\log C_\beta+\beta\log n}{\log n} =\beta.
	\]
	Letting $\beta\uparrow1$ yields $\liminf_n\log|S_n|/\log n\ge1$ almost surely, and combined with the upper bound the limit equals $1$.
\end{proof}

\begin{remark}\label{rem:mrs-upper}
	For $11/16<p<7/8$, \cite[Theorem~2.5(c)]{MaulikRoySadhukhan2025} states that
	\[
		n^{y_p}\left(\frac{S_n}{n}-\Lambda_p\right)\longrightarrow L_p
		\quad\as, \qquad y_p=\frac{7-8p}{3},
	\]
	with $\Lambda_p=0$ in this regime; equivalently, $S_n/n^{\alpha(p)}\to L_p$ almost surely with $\alpha(p)=1-y_p$. That theorem does not assert $L_p\ne0$. By uniqueness of almost-sure limits, its random variable is the same as the one in \cref{prop:nondegenerate-limit}; the proposition therefore strengthens the quoted result by proving $\Pp(L_p=0)=0$.
\end{remark}

\begin{proof}[Proof of \cref{thm:main}]
	By \cref{thm:recurrence} the walk is recurrent for every $p<11/16$, and by \cref{thm:open-interval} it is transient for
	\[
		\frac{11}{16}<p\le\frac78 .
	\]
	By \cite[Proposition~2.8]{MaulikRoySadhukhan2025}, the walk is recurrent at the endpoint $p=11/16$ and transient for $p>7/8$. Hence the walk is recurrent if and only if $p\le11/16$.
\end{proof}

\begin{proof}[Proof of \cref{thm:escape-scale}]
	If
	\[
		\frac{11}{16}<p\le\frac78,
	\]
	then \cref{thm:open-interval} gives
	\[
		|S_n|\longrightarrow\infty 	\quad\as,
	\]
	while \cref{cor:escape-exponent-subballistic} gives
	\[
		\lim_{n\to\infty} \frac{\log|S_n|}{\log n} 	= \frac{8p-4}{3} \quad\as.
	\]
	Since $(8p-4)/3\le1$ in this range, this is precisely \eqref{eq:escape-exponent-main}. It remains to consider $p>7/8$. By \cite[Theorem~2.1]{MaulikRoySadhukhan2025},
	\[
		\frac{S_n}{n}\longrightarrow\Lambda_p 	\quad\as,
	\]
	where
	\[
		\Lambda_p\in\{-c_p,c_p\} \quad\text{with equal probabilities}, \qquad
		c_p= \frac{\sqrt{32p^2-52p+21}}{(2p-1)^2}>0.
	\]
	Consequently,
	\[
		\frac{|S_n|}{n}\longrightarrow c_p>0 \quad\as.
	\]
	In particular, $|S_n|\to\infty$ almost surely, and, for all sufficiently large $n$,
	\[
		\frac{\log|S_n|}{\log n} = 	1+ 	\frac{\log(|S_n|/n)}{\log n} \longrightarrow1 \quad\as.
	\]
	Since $(8p-4)/3>1$ for $p>7/8$, we have
	\[
		1=\min\left\{\frac{8p-4}{3},1\right\}.
	\]
	Thus \eqref{eq:escape-exponent-main} holds throughout the transient regime, which proves part~\textup{(i)}. Part~\textup{(ii)} is precisely \cref{prop:nondegenerate-limit}.
\end{proof}

\section{Asymptotics at the threshold $p=7/8$}\label{sec:critical-scale}
Set $p=7/8$ and define
\[
	x_n:=\frac{S_n}{n}, \qquad y_n:=\frac{A_n}{n}-\frac23, \qquad \lambda:=\frac{\mu^2}{2}=\frac{9}{32}, \qquad z_n:=y_n-\lambda x_n^2.
\]
Write
\[
	\xi_{n+1}:=X_{n+1}-\E(X_{n+1}\mid\F_n), \qquad \eta_{n+1}:=I_{n+1}-\E(I_{n+1}\mid\F_n).
\]
At $p=7/8$, the recursion \eqref{eq:SA} becomes
\begin{align}
	x_{n+1}
	 & =x_n+\frac1{n+1}
	\left(-\frac34x_ny_n+\xi_{n+1}\right),
	\label{eq:critical-x-recursion} \\
	y_{n+1}
	 & =y_n+\frac1{n+1}
	\left(-y_n-\frac32y_n^2+\lambda x_n^2+\eta_{n+1}\right).
	\label{eq:critical-y-recursion}
\end{align}
By \cref{thm:lln}, $(x_n,y_n)\to(0,0)$ almost surely, and hence also $z_n\to0$ almost surely. The variable $z_n$ measures the deviation of $y_n$ from the parabola $y=\lambda x^2$. The choice of $\lambda$ and the derivation of the recursion for $z_n$ are explained in \cref{app:critical-transverse}.

\begin{lemma}\label{lem:transverse-decay}
	For every $\theta\in(0,1)$,
	\begin{equation}\label{eq:transverse-rate}
		n^\theta z_n^2\longrightarrow0 \quad\as.
	\end{equation}
	In particular,
	\[
		z_n=o\!\left(\frac1{\log n}\right) \quad\as.
	\]
\end{lemma}
The proof is given in \cref{app:critical-transverse}. 

\begin{proof}[Proof of \cref{thm:critical-scale}]
	By \cref{thm:open-interval}, \(|S_n|\longrightarrow\infty \ \ \as \) 	at $p=7/8$. Since the increments belong to $\{-1,0,1\}$, a change of sign of $S_n$ forces a visit to the origin. Hence the sign of $S_n$ is
	eventually constant. Let $\Xi\in\{-1,1\}$ denote this eventual sign. From \eqref{eq:critical-x-recursion} and
	\[
		y_n=\lambda x_n^2+z_n, 	\qquad \lambda=\frac{9}{32},
	\]
	we obtain
	\begin{equation}\label{eq:critical-cubic-recursion}
		x_{n+1} = x_n+\frac1{n+1} \left( -\frac{27}{128}x_n^3 -\frac34x_nz_n +\xi_{n+1} \right) =:x_n+\Delta x_n.
	\end{equation}
	In particular, $|\Delta x_n|\le C_0/(n+1)$ for some deterministic constant $C_0$.

	\emph{Step 1.} 	Fix
	\[
		0<\varepsilon<\frac16, \qquad 4\varepsilon<\theta<1.
	\]
	By \cref{prop:lower-escape} with $\beta=1-\varepsilon$ and \cref{lem:transverse-decay}, almost surely there exist a rational $c>0$ and an integer $M$ such that
	\begin{equation}\label{eq:x-lower-critical}
		|x_n|\ge cn^{-\varepsilon}, \qquad |z_n|\le n^{-\theta/2} \qquad\text{for all }n\ge M.
	\end{equation}
	For rational $c>0$ and $M\ge2$, define
	\[
		\varsigma_{c,M} := \inf\left\{ n\ge M: |x_n|<cn^{-\varepsilon} \ \text{or}\
		|z_n|>n^{-\theta/2} \right\},
	\]
	and let
	\[
		G_n:=\{n<\varsigma_{c,M}\}.
	\]
	Then $G_n\in\F_n$, and \eqref{eq:x-lower-critical} implies
	\[
		\Pp\left( \bigcup_{c,M}\{\varsigma_{c,M}=\infty\} \right)=1,
	\]
	where the union is over rational $c>0$ and integers $M\ge2$. Therefore it is enough to prove
	\begin{equation}\label{eq:x-sharp-critical}
		x_n^2\log n\longrightarrow\frac{64}{27}
	\end{equation}
	almost surely on each event $\{\varsigma_{c,M}=\infty\}$. Fix such a pair $(c,M)$. Since on $G_n$,
	\[
		|x_n|\ge cn^{-\varepsilon},
	\]
	we have
	\[
		\frac{|\Delta x_n|}{|x_n|} \le C_0c^{-1}n^{-1+\varepsilon} \longrightarrow0.
	\]
	Thus, after increasing $M$ if necessary, we may assume that
	\[
		|\Delta x_n|\le\frac12|x_n| \qquad\text{on }G_n,\quad n\ge M.
	\]
	This does not change the argument, since $\varsigma_{c,M}=\infty$ implies $\varsigma_{c,M'}=\infty$ for every $M'\ge M$.

	\emph{Step 2.} On $G_n$,
	\[
		|x_{n+1}| \ge |x_n|-|\Delta x_n| \ge \frac12|x_n|>0.
	\]
	Hence Taylor's formula for $f(x)=x^{-2}$ gives
	\begin{equation}\label{eq:inverse-square-expansion}
		\frac1{x_{n+1}^2}-\frac1{x_n^2} = \frac{27}{64}\frac1{n+1} +\frac32\frac{z_n}{(n+1)x_n^2} -\frac{2\xi_{n+1}}{(n+1)x_n^3} +\rho_{n+1}.
	\end{equation}
	The remainder has the form
	\[
		\rho_{n+1} 	= \frac12f''(\vartheta_n)\,(\Delta x_n)^2
	\]
	for some $\vartheta_n$ between $x_n$ and $x_{n+1}$. Since
	$|\vartheta_n|\ge|x_n|/2$,
	\[
		\begin{aligned}
			|\rho_{n+1}| \le \frac{3(\Delta x_n)^2}{(|x_n|/2)^4} = \frac{48(\Delta x_n)^2}{|x_n|^4} \le \frac{C_1}{(n+1)^2}|x_n|^{-4} \le C_2n^{-2+4\varepsilon}.
		\end{aligned}
	\]
	Because $4\varepsilon<1$,
	\begin{equation}\label{eq:rho-summable}
		\sum_{n\ge M}\1_{G_n}|\rho_{n+1}|<\infty.
	\end{equation}
    
		\emph{Step 3.}
	Since $|X_{n+1}|\le1$,
	\[
		\E(\xi_{n+1}^2\mid\F_n) = \operatorname{Var}(X_{n+1}\mid\F_n) \le1.
	\]
	Both $\1_{G_n}$ and $x_n$ are $\F_n$-measurable, and $|x_n|\ge cn^{-\varepsilon}$
	on $G_n$, so each summand of
	\[
		\sum_{n\ge M} \1_{G_n} \frac{2\xi_{n+1}}{(n+1)x_n^3}
	\]
	is a bounded martingale difference, and the partial sums form a
	martingale. Its predictable quadratic variation satisfies
	\[
		\sum_{n\ge M} \1_{G_n} \frac{4\,\E(\xi_{n+1}^2\mid\F_n)}{(n+1)^2x_n^6}
		\le \sum_{n\ge M} \1_{G_n} \frac{4}{(n+1)^2}|x_n|^{-6} \le 	C_3\sum_{n\ge M}n^{-2+6\varepsilon}.
	\]
	Since $\varepsilon<1/6$, the series on the right converges. Hence the
	martingale converges almost surely, by
	\cite[Theorem~2.17]{HallHeyde1980}.

	For the term involving $z_n$, on $G_n$,
	\[
		\frac32\frac{|z_n|}{(n+1)x_n^2} \le \frac{3}{2c^2}\, n^{-1- \theta/2+2\varepsilon}.
	\]
	Since $\theta>4\varepsilon$,
	\[
		-1-\frac{\theta}{2}+2\varepsilon<-1,
	\]
	and therefore
	\begin{equation}\label{eq:z-term-summable}
		\sum_{n\ge M} \1_{G_n} \frac{3|z_n|}{2(n+1)x_n^2} <\infty.
	\end{equation}
    
	\emph{Step 4.}
	On $\{\varsigma_{c,M}=\infty\}$, one has $G_n$ for every $n\ge M$. The three error terms in \eqref{eq:inverse-square-expansion} were controlled in Steps 2 and 3: the series $\sum_n\rho_{n+1}$ and $\sum_n\tfrac32z_n/((n+1)x_n^2)$ converge absolutely by \eqref{eq:rho-summable} and \eqref{eq:z-term-summable}, and the martingale $\sum_n2\xi_{n+1}/((n+1)x_n^3)$ converges almost surely. Hence
	the partial sums of the right-hand side of \eqref{eq:inverse-square-expansion}, with the leading term removed, converge almost surely to a finite random limit. Moreover,
	\[
		\sum_{n=M}^{N-1}\frac1{n+1} = \log N+C_M+o(1) \qquad(N\to\infty),
	\]
	where $C_M$ is a deterministic constant depending only on $M$. Therefore, summing \eqref{eq:inverse-square-expansion} from $M$ to $N-1$ gives
	\begin{equation}\label{eq:Theta-existence}
		\frac1{x_N^2}-\frac{27}{64}\log N
		\quad\text{converges almost surely to a finite limit, }\Theta .
	\end{equation}
	This proves \eqref{eq:Theta} on $\{\varsigma_{c,M}=\infty\}$; since the union of these events has probability one, and since the limit in \eqref{eq:Theta-existence} does not depend on $(c,M)$, the random variable
	$\Theta$ is well defined almost surely and \eqref{eq:Theta} holds almost surely. Dividing by $\log N$ we obtain in particular
	\begin{equation}\label{eq:x-sharp-critical-bis}
		x_N^2\log N =\frac{\log N}{\frac{27}{64}\log N+\Theta+o(1)} \longrightarrow
		\frac{64}{27},
	\end{equation}
	which is \eqref{eq:x-sharp-critical}.
    
	\emph{Step 5.}
	Since the sign of $x_n$ is eventually $\Xi$,
	\[
		\sqrt{\log n}\,x_n 	= \Xi\sqrt{x_n^2\log n} \longrightarrow \frac{8\sqrt3}{9}\,\Xi \quad\as.
	\]
	As $x_n=S_n/n$, this proves \eqref{eq:critical-position-scale}. The model and its initial law are invariant under the global sign change
	\[
		(X_j)_{j\ge1}\longmapsto(-X_j)_{j\ge1},
	\]
	which sends $\Xi$ to $-\Xi$. Therefore
	\[
		\Pp(\Xi=1)=\Pp(\Xi=-1)=\frac12.
	\]
	Finally, $y_n=\lambda x_n^2+z_n$ with $\lambda=9/32$. By \eqref{eq:Theta},
	\[
		\lambda x_n^2 =\frac{9/32}{\frac{27}{64}\log n+\Theta+o(1)} =\frac23\cdot\frac1{\log n+\frac{64}{27}\Theta+o(1)}.
	\]
	By
	\cref{lem:transverse-decay} we have $z_n=o(n^{-\theta/2})$ for every
	$\theta\in(0,1)$, and hence $z_n(\log n)^2\to0$. It follows that
	\[
		y_n =\frac23\cdot \frac1{\log n+\frac{64}{27}\Theta+o(1)},
	\]
	which is \eqref{eq:critical-activity-scale}. Multiplying by $\log n$ gives
	$\log n\,y_n\to2/3$. Expanding, one gets
	\[
		y_n =\frac{2}{3\log n} \left(1+\frac{\frac{64}{27}\Theta+o(1)}{\log n}\right)^{-1} =\frac{2}{3\log n} -\frac{128\,\Theta}{81}\cdot\frac1{(\log n)^2} +o\!\left(\frac1{(\log n)^2}\right).
	\]
	Since $y_n=A_n/n-2/3$, this is the last assertion of the theorem.
\end{proof}
These estimates also determine the asymptotic relation between the original time index and the number of nonzero steps. As a consequence, we obtain the corresponding scale for the embedded walk at $p=7/8$.

\begin{corollary}\label{cor:critical-embedded-scale}
	At $p=7/8$,
	\begin{align}
	& \tau_k
		  =\frac32k-\frac{3k}{2\log k} 	+o\!\left(\frac{k}{\log k}\right),
		\label{eq:tau-critical-expansion}     \\
	&	\sqrt{\log k}\,\frac{\widehat S_k}{k} \longrightarrow\frac{4\sqrt3}{3}\,\Xi
		  \quad\as.
		\label{eq:embedded-critical-scale}
	\end{align}
\end{corollary}

\begin{proof}
	Since $A_{\tau_k}=k$, \eqref{eq:critical-activity-scale} gives
	\[
		k=\frac{2\tau_k}{3} \left(1+\frac1{\log\tau_k} +o\!\left(\frac1{\log\tau_k}\right)\right).
	\]
	As $\tau_k/k\to3/2$, inversion yields \eqref{eq:tau-critical-expansion}. Substituting $n=\tau_k$ into \eqref{eq:critical-position-scale}, and using $\tau_k/k\to3/2$ and $\log\tau_k/\log k\to1$, gives \eqref{eq:embedded-critical-scale}.
\end{proof}

\section{Conclusion}\label{sec:conclusion}

We have completed the recurrence--transience classification for the elephant random walk with two memory channels. The walk is recurrent for $p\le11/16$ and transient for $p>11/16$. In particular, this settles the
range $11/16<p\le7/8$ left open in \cite{MaulikRoySadhukhan2025}.

For $11/16<p<7/8$, we obtained a more precise description of the transient walk. There exists a finite random variable $L_p$ such that
\[
	\frac{S_n}{n^{(8p-4)/3}} \longrightarrow L_p \quad\as, \qquad \Pp(L_p=0)=0.
\]
Consequently, $|S_n|$ grows on the scale $n^{(8p-4)/3}$. Since $(8p-4)/3$ lies between $1/2$ and $1$, the walk escapes faster than the diffusive scale $\sqrt n$ but slower than the linear scale $n$.

At $p=7/8$, the walk remains sub-ballistic but transient. More precisely,
\[
	\sqrt{\log n}\,\frac{S_n}{n} \longrightarrow \frac{8\sqrt3}{9}\,\Xi \quad\as,
\]
where $\Xi$ takes the values $1$ and $-1$ with equal probability. Thus the walk escapes at the scale $n/\sqrt{\log n}$. We also proved that
\[
	\log n\left(\frac{A_n}{n}-\frac23\right) \longrightarrow\frac23 \quad\as,
\]
which gives the rate at which the proportion of nonzero increments approaches $2/3$.

Several questions remain open. For $11/16<p<7/8$, we do not determine the distribution of $L_p$ or whether it has a density. At $p=7/8$, the distribution and moments of the random variable $\Theta$ in \cref{thm:critical-scale} are also unknown. For $11/16<p<7/8$, it would also be interesting to obtain a limit theorem for the fluctuations of $S_n$ around $L_p n^{(8p-4)/3}$.

Finally, natural extensions include models with $m\ge3$ memory channels and their multidimensional versions. In higher dimensions, one may ask whether the recurrence--transience thresholds found in \cite{Qin2025} are connected
to the corresponding multi-channel models through an analogous effective-parameter relation.

\appendix
\crefalias{section}{appendix}
\crefname{appendix}{Appendix}{Appendices}
\Crefname{appendix}{Appendix}{Appendices}

\section{An auxiliary convergence lemma}\label{app:slow-variation}
This appendix proves the convergence criterion used in the main argument. The lemma is deterministic; in the main text it is applied pathwise, on the almost sure event where its hypotheses hold.

\begin{lemma}\label{lem:slow-variation}
	Let $(u_n)$ be a nonnegative sequence of real numbers such that
	\[
		|u_{n+1}-u_n|\le \frac{C}{n+1}
	\]
	for some constant $C\in(0,\infty)$. If
	\[
		\sum_{n=1}^{\infty}\frac{u_n}{n+1}<\infty,
	\]
	then $u_n\to0$.
\end{lemma}
\begin{proof}
	Suppose, for contradiction, that $u_n\not\to0$. Then there are $\varepsilon>0$ and infinitely many indices $n$ with $u_n\ge\varepsilon$. The increment bound prevents $u$ from dropping quickly after such an index.
	Set
	\[
		\Lambda:=e^{\varepsilon/(2C)}>1,
	\]
	and let $n$ be an index with $u_n\ge\varepsilon$. For any $j$ with $n\le j\le\Lambda n$ we have, telescoping the increments and using $|u_{i+1}-u_i|\le C/(i+1)$,
	\[
		|u_j-u_n| \le\sum_{i=n}^{j-1}\frac{C}{i+1} \le C\int_{n}^{j}\frac{\dd s}{s}
		=C\log\frac jn \le C\log\Lambda =\frac{\varepsilon}{2},
	\]
	so
	\begin{equation}\label{eq:slowvar-lower}
		u_j\ge u_n-\frac{\varepsilon}{2}\ge\frac{\varepsilon}{2}, \qquad n\le j\le\Lambda n.
	\end{equation}
	Thus each index at which $u$ reaches $\varepsilon$ forces $u$ to stay above $\varepsilon/2$ on a whole multiplicative interval $[n,\Lambda n]$. Each such interval contributes a fixed amount to the series. Indeed, by
	\eqref{eq:slowvar-lower},
	\[
		\sum_{n\le j\le\Lambda n}\frac{u_j}{j+1} \ge\frac{\varepsilon}{2}\sum_{n\le j\le\Lambda n}\frac1{j+1} \ge\frac{\varepsilon}{2}\int_{n+1}^{\Lambda n}\frac{\dd s}{s} =\frac{\varepsilon}{2}\left(\log\Lambda-\log\frac{n+1}{n}\right) \longrightarrow\frac{\varepsilon}{2}\log\Lambda =\frac{\varepsilon^2}{4C}
	\]
	as $n\to\infty$; in particular the left-hand side is at least $\varepsilon^2/(8C)$ once $n$ is large enough.
	Finally, since there are infinitely many indices with $u_n\ge\varepsilon$, we may extract from them a sequence $n_1<n_2<\cdots$ with $n_{i+1}>\Lambda n_i$ for every $i$, so that the corresponding intervals $[n_i,\Lambda n_i]$ are pairwise disjoint. Summing the contributions of these disjoint intervals gives
	\[
		\sum_{j\ge1}\frac{u_j}{j+1}
		\ge\sum_{i\ge i_0}\frac{\varepsilon^2}{8C}
		=\infty,
	\]
	contradicting the hypothesis $\sum_j u_j/(j+1)<\infty$. Hence $u_n\to0$.
\end{proof}

\section{Polynomial decay of $z_n$}
\label{app:critical-transverse}

This appendix proves the estimate on $z_n$ stated in \cref{lem:transverse-decay}. We use the notation introduced at the beginning of \cref{sec:critical-scale}.

\begin{proof}[Proof of \cref{lem:transverse-decay}]
	We first explain the choice of $\lambda$. For an arbitrary constant $c$, define
	\[
		z_n^{(c)}:=y_n-cx_n^2.
	\]
	Using \eqref{eq:critical-x-recursion} and \eqref{eq:critical-y-recursion}, a direct calculation gives
	\begin{align}
		z_{n+1}^{(c)}-z_n^{(c)}
		={}&
		\frac1{n+1}
		\left[
			\left(\frac{\mu^2}{2}-c\right)x_n^2
			-\left(1+\frac32y_n\right)z_n^{(c)}
			+\eta_{n+1}-2cx_n\xi_{n+1}
		\right]
		\notag\\
		&-
		\frac{c}{(n+1)^2}
		\left(-\frac34x_ny_n+\xi_{n+1}\right)^2 .
		\label{eq:z-c-recursion}
	\end{align}
	The choice
	\[
		c=\frac{\mu^2}{2}
	\]
	is therefore the unique one that removes the term of order $x_n^2$ which is independent of $z_n^{(c)}$. At $p=7/8$, this gives
	\[
		\lambda:=\frac{\mu^2}{2}=\frac{9}{32}, \qquad z_n:=y_n-\lambda x_n^2.
	\]
	With this choice, \eqref{eq:z-c-recursion} becomes
	\begin{equation}\label{eq:z-recursion}
		z_{n+1} = z_n+\frac1{n+1} \left[ -\left(1+\frac{27}{64}x_n^2+\frac32z_n\right)z_n 	+\zeta_{n+1} \right] +R_{n+1},
	\end{equation}
	where we used $y_n=\lambda x_n^2+z_n$ and set
	\[
		\zeta_{n+1} := 	\eta_{n+1}-\frac{9}{16}x_n\xi_{n+1},
	\]
	and
	\[
		R_{n+1} := 	-\frac{\lambda}{(n+1)^2} \left(- \frac34x_ny_n+\xi_{n+1}\right)^2.
	\]
	Since $x_n$ is $\F_n$-measurable, $(\zeta_{n+1})$ is a bounded martingale-difference sequence. Moreover, all the variables in the last square are bounded, so there is a deterministic constant $C_0$ such that
	\begin{equation}\label{eq:R-bound}
		|R_{n+1}| \le \frac{C_0}{(n+1)^2}.
	\end{equation}

	Fix $\theta\in(0,1)$ and let $\varepsilon_0:=1/6$. For each integer $M\ge2$, define
	\[
		\varsigma_M := \inf\left\{n\ge M: |x_n|>\varepsilon_0 \ \text{or}\ |z_n|>\varepsilon_0 \right\}.
	\]
	On $\{n<\varsigma_M\}$, set
	\[
		\chi_n := 1+\frac{27}{64}x_n^2+\frac32z_n.
	\]
	Then
	\[
		\frac34\le\chi_n\le2, \qquad 2\chi_n-1\ge\frac12.
	\]
	Thus, on $\{n<\varsigma_M\}$, \eqref{eq:z-recursion} can be written as
	\[
		z_{n+1} = \left(1-\frac{\chi_n}{n+1}\right)z_n +\frac{\zeta_{n+1}}{n+1} +R_{n+1}.
	\]

	Squaring and taking conditional expectations, and then using \eqref{eq:R-bound} together with the boundedness of $\zeta_{n+1}$, gives
	\begin{equation}\label{eq:z-square-pre}
		\E(z_{n+1}^2\mid\F_n) \le \left( 1-\frac{2\chi_n}{n+1} +\frac{C_1}{(n+1)^2}
		\right)z_n^2 +\frac{C_1}{(n+1)^2}
	\end{equation}
	on $\{n<\varsigma_M\}$, for some deterministic constant $C_1$. Here the cross term involving $\zeta_{n+1}$ has conditional expectation zero, while the terms involving $R_{n+1}$ are controlled by
	\eqref{eq:R-bound}.

	For all sufficiently large $n$,
	\[
		\frac{C_1}{n+1} \le\frac12 \le2\chi_n-1.
	\]
	It follows from \eqref{eq:z-square-pre} that
	\begin{equation}\label{eq:z-square-contraction}
		\E(z_{n+1}^2\mid\F_n) \le \left(1-\frac1{n+1}\right)z_n^2 +\frac{C_2}{(n+1)^2}
	\end{equation}
	on $\{n<\varsigma_M\}$, for all sufficiently large $n$.

	Define
	\[
		u_n := (n+1)^\theta z_n^2 \1_{\{n<\varsigma_M\}}.
	\]
	Since
	\[
		\{n+1<\varsigma_M\}\subseteq\{n<\varsigma_M\},
	\]
	\eqref{eq:z-square-contraction} implies
	\[
		\E(u_{n+1}\mid\F_n) \le (n+2)^\theta \left[ \left(1-\frac1{n+1}\right)z_n^2
			+\frac{C_2}{(n+1)^2} \right] \1_{\{n<\varsigma_M\}}.
	\]
	Because $0<\theta<1$, concavity gives
	\[
		(n+2)^\theta \le (n+1)^\theta \left(1+\frac{\theta}{n+1}\right).
	\]
	Furthermore,
	\[
		\left(1+\frac{\theta}{n+1}\right) \left(1-\frac1{n+1}\right) \le 1-\frac{1-\theta}{n+1}.
	\]
	Consequently,
	\[
		\E(u_{n+1}\mid\F_n) \le u_n-\frac{1-\theta}{n+1}u_n +\frac{C_3}{(n+1)^{2-\theta}}.
	\]
	Since $2-\theta>1$, the last term is summable. The Robbins--Siegmund theorem~\cite{RobbinsSiegmund1971} therefore implies that $u_n$ converges almost surely and
	\[
		\sum_n\frac{u_n}{n+1}<\infty \quad\as.
	\]
	The limit of $u_n$ must be zero. Indeed, if it were strictly positive, then $u_n$ would eventually be bounded below by a positive constant, contradicting the convergence of the last series. Hence, on
	$\{\varsigma_M=\infty\}$,
	\[
		n^\theta z_n^2\longrightarrow0 \quad\as.
	\]

	It remains to remove the stopping time. By \cref{thm:lln},
	\[
		x_n\longrightarrow0, \qquad y_n\longrightarrow0 \quad\as,
	\]
	and therefore $z_n=y_n-\lambda x_n^2\to0$ almost surely. Thus, for almost every sample path, there exists $M_0$ such that
	\[
		|x_n|\le\varepsilon_0, \qquad |z_n|\le\varepsilon_0 \qquad\text{for all }n\ge M_0.
	\]
	For every $M\ge M_0$, we then have $\varsigma_M=\infty$. Hence
	\[
		\Pp\left( \bigcup_{M\ge2}\{\varsigma_M=\infty\} \right)=1,
	\]
	and the preceding estimate yields
	\[
		n^\theta z_n^2\longrightarrow0 \quad\as.
	\]
	This proves \eqref{eq:transverse-rate}.

	Finally,
	\[
		|z_n|=o(n^{-\theta/2}) \quad\as.
	\]
	Since $n^{-\theta/2}=o(1/\log n)$ for every fixed $\theta\in(0,1)$, we conclude that
	\[
		z_n=o\!\left(\frac1{\log n}\right) 	\quad\as.
	\]
\end{proof}


\begin{thebibliography}{99}
	\bibitem{BaurBertoin2016}
	E.~Baur and J.~Bertoin.
	\newblock Elephant random walks and their connection to P\'olya-type urns.
	\newblock \emph{Phys. Rev. E}, 94:052134, 2016.
	\bibitem{Bercu2018}
	B.~Bercu.
	\newblock A martingale approach for the elephant random walk.
	\newblock \emph{J. Phys. A}, 51(1):015201, 2018.
	\bibitem{Bercu2022stops}
	B.~Bercu.
	\newblock On the elephant random walk with stops playing hide and seek with the Mittag--Leffler distribution.
	\newblock \emph{J. Stat. Phys.}, 189(1):Paper No.~12, 2022.
	\bibitem{Bercu2025stops}
	B.~Bercu.
	\newblock On the multidimensional elephant random walk with stops.
	\newblock \emph{Stochastic Process. Appl.}, 189:104692, 2025.
	\bibitem{BercuLaulin2019}
	B.~Bercu and L.~Laulin.
	\newblock On the multi-dimensional elephant random walk.
	\newblock \emph{J. Stat. Phys.}, 175(6):1146--1163, 2019.
	\bibitem{Bertenghi2022}
	M.~Bertenghi.
	\newblock Functional limit theorems for the multi-dimensional elephant random walk.
	\newblock \emph{Stoch. Models}, 38(1):37--50, 2022.
	\bibitem{Bertoin2021}
	J.~Bertoin.
	\newblock Scaling exponents of step-reinforced random walks.
	\newblock \emph{Probab. Theory Related Fields}, 179(1):295--315, 2021.
	\bibitem{Bertoin2022}
	J.~Bertoin.
	\newblock Counting the zeros of an elephant random walk.
	\newblock \emph{Trans. Amer. Math. Soc.}, 375(8):5539--5560, 2022.
	\bibitem{Businger2018}
	S.~Businger.
	\newblock The shark random swim (L\'evy flight with memory).
	\newblock \emph{J. Stat. Phys.}, 172(3):701--717, 2018.
	\bibitem{ChenLaulin2023}
	J.~Chen and L.~Laulin.
	\newblock Analysis of the smoothly amnesia-reinforced multidimensional elephant random walk.
	\newblock \emph{J. Stat. Phys.}, 190(10):Paper No.~158, 2023.
	\bibitem{ColettiGavaSchutz2017a}
	C.~F. Coletti, R.~Gava, and G.~M. Sch\"utz.
	\newblock Central limit theorem and related results for the elephant random walk.
	\newblock \emph{J. Math. Phys.}, 58(5):053303, 2017.
	\bibitem{ColettiGavaSchutz2017b}
	C.~F. Coletti, R.~Gava, and G.~M. Sch\"utz.
	\newblock A strong invariance principle for the elephant random walk.
	\newblock \emph{J. Stat. Mech. Theory Exp.}, 2017(12):123207, 2017.
	\bibitem{ColettiPapageorgiou2021}
	C.~F. Coletti and I.~Papageorgiou.
	\newblock Asymptotic analysis of the elephant random walk.
	\newblock \emph{J. Stat. Mech. Theory Exp.}, 2021(1):013205, 2021.
	\bibitem{CressoniSilvaViswanathan2007}
	J.~C. Cressoni, M.~A.~A. da~Silva, and G.~M. Viswanathan.
	\newblock Amnestically induced persistence in random walks.
	\newblock \emph{Phys. Rev. Lett.}, 98:070603, 2007.
	\bibitem{CurienLaulin2024}
	N.~Curien and L.~Laulin.
	\newblock Recurrence of the plane elephant random walk.
	\newblock \emph{C. R. Math. Acad. Sci. Paris}, 362:1183--1188, 2024.
	\bibitem{DedeckerFanHuMerlevede2023}
	J.~Dedecker, X.~Fan, H.~Hu, and F.~Merlev\`ede.
	\newblock Rates of convergence in the central limit theorem for the elephant random walk with random step sizes.
	\newblock \emph{J. Stat. Phys.}, 190(10):Paper No.~154, 2023.
	\bibitem{FanHuMa2021}
	X.~Fan, H.~Hu, and X.~Ma.
	\newblock Cram\'er moderate deviations for the elephant random walk.
	\newblock \emph{J. Stat. Mech. Theory Exp.}, 2021(2):023402, 2021.
	\bibitem{GuerinLaulinRaschel}
	H.~Gu\'erin, L.~Laulin, and K.~Raschel.
	\newblock A fixed-point equation approach for the superdiffusive elephant random walk.
	\newblock \emph{Ann. Inst. Henri Poincar\'e Probab. Stat.}, 62(2):973--1005, 2026.
	\bibitem{GuerinLaulinRaschelSimon}
	H.~Gu\'erin, L.~Laulin, K.~Raschel, and T.~Simon.
	\newblock On the limit law of the superdiffusive elephant random walk.
	\newblock \emph{Electron. J. Probab.}, 30, 2025.

    
	\bibitem{GutStadtmuller2021}
	A.~Gut and U.~Stadtm\"uller.
	\newblock Variations of the elephant random walk.
	\newblock \emph{J. Appl. Probab.}, 58(3):805--829, 2021.
	\bibitem{GutStadtmuller2022}
	A.~Gut and U.~Stadtm\"uller.
	\newblock The elephant random walk with gradually increasing memory.
	\newblock \emph{Statist. Probab. Lett.}, 189:109598, 2022.
	\bibitem{HallHeyde1980}
	P.~Hall and C.~C. Heyde.
	\newblock \emph{Martingale Limit Theory and Its Application}.
	\newblock Academic Press, New York, 1980.
	\bibitem{Janson2004}
	S.~Janson.
	\newblock Functional limit theorems for multitype branching processes and generalized P\'olya urns.
	\newblock \emph{Stochastic Process. Appl.}, 110(2):177--245, 2004.
	\bibitem{KubotaTakei2019}
	N.~Kubota and M.~Takei.
	\newblock Gaussian fluctuation for superdiffusive elephant random walks.
	\newblock \emph{J. Stat. Phys.}, 177(6):1157--1171, 2019.
	\bibitem{KumarHarbolaLindenberg2010}
	N.~Kumar, U.~Harbola, and K.~Lindenberg.
	\newblock Memory-induced anomalous dynamics: Emergence of diffusion, subdiffusion, and superdiffusion from a single random walk model.
	\newblock \emph{Phys. Rev. E}, 82:021101, 2010.
	\bibitem{Kursten2016}
	R.~K\"ursten.
	\newblock Random recursive trees and the elephant random walk.
	\newblock \emph{Phys. Rev. E}, 93:032111, 2016.
	\bibitem{Laulin2022}
	L.~Laulin.
	\newblock Introducing smooth amnesia to the memory of the elephant random walk.
	\newblock \emph{Electron. Commun. Probab.}, 27:1-12, 2022.
	\bibitem{MaulikRoySadhukhan2025}
	K.~Maulik, P.~Roy, and T.~Sadhukhan.
	\newblock Phase transitions for elephant random walks with two memory channels.
	\newblock arXiv:2509.10225, 2025.
	\bibitem{MiyazakiTakei2020}
	T.~Miyazaki and M.~Takei.
	\newblock Limit theorems for the `laziest' minimal random walk model of elephant type.
	\newblock \emph{J. Stat. Phys.}, 181(2):587--602, 2020.
	\bibitem{ParaanEsguerra2006}
	F.~N.~C. Paraan and J.~P. Esguerra.
	\newblock Exact moments in a continuous time random walk with complete memory of its history.
	\newblock \emph{Phys. Rev. E}, 74:032101, 2006.
	\bibitem{Qin2025}
	S.~Qin.
	\newblock Recurrence and transience of multidimensional elephant random walks.
	\newblock \emph{Ann. Probab.}, 53(3):1049--1078, 2025.
	\bibitem{RobbinsSiegmund1971}
	H.~Robbins and D.~Siegmund.
	\newblock A convergence theorem for nonnegative almost supermartingales and some applications.
	\newblock In \emph{Optimizing Methods in Statistics}, pages 233--257. Academic Press, 1971.
	\bibitem{Saha2022}
	S.~Saha.
	\newblock Random walk with multiple memory channels.
	\newblock \emph{Phys. Rev. E}, 106:L062105, 2022.
	\bibitem{SchutzTrimper2004}
	G.~M. Sch\"utz and S.~Trimper.
	\newblock Elephants can always remember: exact long-range memory effects in a non-Markovian random walk.
	\newblock \emph{Phys. Rev. E}, 70:045101, 2004.
\end{thebibliography}
\end{document}